\documentclass[reqno]{amsart}

\usepackage[T1]{fontenc}
\usepackage{graphicx}
\usepackage{newtxtext}
\usepackage[varg]{newtxmath}

\usepackage{pinlabel}
\usepackage{amsmath}
\usepackage{amsfonts}
\usepackage{amssymb}
\usepackage{enumerate}
\usepackage{expdlist}

\usepackage{tikz}
\usetikzlibrary{decorations.pathreplacing}
\usetikzlibrary{calc}
\usetikzlibrary{shapes.geometric}
\usetikzlibrary{angles}
\usetikzlibrary{positioning}
\usepackage{empheq}
\usepackage{tkz-euclide}
\usepackage{pgfplots}
\usepgfplotslibrary{ternary}
\usetkzobj{all}

\usepackage[nooneline,hang]{subfigure}

\usepackage[margin=0pt,font=normalsize,singlelinecheck=off, 
labelfont=bf]{caption}
\usepackage{graphicx}
\usepackage{bbold}

\usepackage{floatflt}

\usepackage{mathrsfs}
\usepackage{url}

\DeclareFontFamily{U}{wncy}{}
    \DeclareFontShape{U}{wncy}{m}{n}{<->wncyr10}{}
    \DeclareSymbolFont{mcy}{U}{wncy}{m}{n}
    \DeclareMathSymbol{\ELL}{\mathord}{mcy}{"4C} 

\newcommand{\R}{\mathbb R}
\newcommand{\C}{\mathbb C}

\newcommand{\Z}{\mathbb Z}

\newcommand{\CP}{\mathbb CP}

\newcommand{\cro}{\mathrm{cr}}
\newcommand{\Ha}{\mathbb{H}}

\theoremstyle{plain}
\newtheorem{theorem}{Theorem}[section]

\newtheorem{corollary}[theorem]{Corollary}
\newtheorem{lemma}[theorem]{Lemma}

\theoremstyle{definition}
\newtheorem{remark}[theorem]{Remark}
\newtheorem{definition}[theorem]{Definition}

\title{Conformally 
symmetric triangular lattices and discrete $\vartheta$-conformal maps}

\author{Ulrike B\"ucking}

\date{\today}

\begin{document}

\begin{abstract}
Two immersed triangulations in the plane with the same combinatorics are 
considered 
as preimage and image of a discrete immersion $F$. We compare the cross-ratios
$Q$ and $q$ of corresponding pairs of adjacent triangles in the two
triangulations. If for every pair the arguments of these cross-ratios (i.e.\ 
intersection angles of circumcircles) agree, $F$ is a discrete conformal map 
based on circle patterns. Similarly, if for every pair the absolute values of 
the corresponding cross-ratios $Q$ and $q$
(i.e.\ length cross-ratios) agree, the two triangulations are discretely 
conformally equivalent. We introduce a new notion, discrete 
$\vartheta$-conformal maps, 
which interpolates between these two known definitions of discrete 
conformality for planar triangulations. We prove that there 
exists an associated variational principle.
In particular, discrete $\vartheta$-conformal maps are unique maximizers of
a locally defined concave functional ${\mathcal F}_\vartheta$ in suitable
variables.
Furthermore, we study  conformally 
symmetric triangular lattices which contain examples of discrete 
$\vartheta$-conformal maps.
\end{abstract}

\maketitle

\section{Introduction}

Conformal maps of planar domains, that is, holomorphic maps with non-vanishing 
derivatives, build an important and classical subject in complex analysis which 
has been intensively studied. Conformality of a map $f$ may be characterized 
by the fact that it infinitesimally preserves cross-ratios. In particular,
$f$ also preserves intersection angles and orientation. 
Moreover, the standard metric $g$ of the complex plane is changed conformally 
by $f$ to $\tilde g$, that is $\tilde g=\text{e}^u g$ for some smooth function 
$u$. Note that all these properties are invariant under M\"obius 
transformations.

In the last decades, a growing interest in discrete conformal maps has emerged. 
Thurston first introduced in~\cite{Thu85} planar circle packings, that is 
configurations of touching discs corresponding to a triangulation, as discrete 
conformal maps. This idea has been extended to circle patterns which allow 
intersecting circles with fixed intersection angles, for instance 
Schramm's orthogonal circle patterns~\cite{Sch97}. Many explicit classes of
examples of circle packings or circle patterns which correspond to special 
conformal maps like polynomials, exponential functions, $z^\gamma,\log, 
\text{erf}$, see~\cite{BDS94,Sch97,Bo99,AB00,BH01,BH03,BMS05}, have been 
discovered using integrable structures or other additional properties.
Furthermore, circle patterns (and circle packings) can be obtained from
variational principles for the radii of the circles, 
see~\cite{CV91,Ri94,BS02,KSSch}. Moreover, two circle patterns with the same 
underlying
combinatorics may be considered as discrete conformal map if all 
intersection angles of corresponding pairs of circles for incident vertices agree,
see for example~\cite{Sch97,KSSch,Bue08}.

Another more recent metric approach led to a different notion of discrete 
conformal maps based on discretely conformally equivalent 
triangulations~\cite{BSSp16,Bue17conv}. This notion was introduced by Luo
in~\cite{Luo}. Here, the 
length cross-ratios (i.e\ absolute values of the cross-ratios) agree for 
corresponding pairs of incident triangles. For these discrete conformal maps 
there is only one explicit class of examples known so far, which corresponds to 
exponential functions, see~\cite{WGS15}. Nevertheless, conformally equivalent 
triangulations may also be obtained from a variational principle for the edge 
lengths, see~\cite{SSchP08,BPS13}. Moreover, there is a relation of this 
functional studied in~\cite{BPS13} to a variational principle for circle 
patterns based on edge lengths.
Note that there also exists a variational principle for more general discrete 
conformal changes of triangulations from a metric viewpoint, see~\cite{G11}.

In this article, we propose a new
definition of discrete conformal maps $F:T\to \widehat{T}$ between two 
immersed planar triangulations $T$ and $\widehat{T}$ with the same combinatorics
which generalizes the two known notions 
based on circle patterns and discrete conformal equivalence. Recall that the 
cross-ratio of the four vertices of two incident triangles encode the 
intersection angle of the 
corresponding circumcircles of the triangles as well as the length cross-ratio.
We may add circumcircles to all triangles of the triangulations $T$ and 
$\widehat{T}$. In 
this way we obtain two {\em circle patterns} with the same combinatorics.
Now $F$ is discrete conformal if all intersection angles of these 
circumcircles agree for corresponding pairs of adjacent triangles. In terms 
of cross-ratios, this means that the arguments of the cross-ratios are the same
for all pairs of adjacent triangles. 
Similarly, focusing on the absolute values of the cross-ratios of all 
corresponding pairs of adjacent 
triangles, we could call $F$ discrete conformal if all length cross-ratios
are preserved.
Discrete $\vartheta$-conformality generalize these two notions of discrete 
conformality by including them into the one parameter family of definitions 
which interpolates smoothly between them. Instead of intersections angles or 
length cross-ratios, we preserve a combination of these values for an 
arbitrary, but fixed parameter $\vartheta$. See Section~\ref{secDef} for more
details.
Figure~\ref{figExvartheta} shows some examples.
There is also a relation to discrete holomorphic quadratic differentials 
defined by Lam and Pinkall in~\cite{Lam2016}, see Subsection~\ref{remquaddiff}.

One main result of this paper is the proof in Section~\ref{SecVari} that the
condition for discrete $\vartheta$-conformal maps, see~\eqref{eqdefdiscconf}, 
is variational which is known to hold  for circle patterns~\cite{BS02} and 
conformally equivalent triangulations~\cite{BPS13}. This means that using 
suitable coordinates, discrete $\vartheta$-conformal maps correspond (locally)
to critical points of a functional. Calculating the second derivative, we see 
the particularly nice property that the functional is concave and the Hessian 
is 
related to the cotan-Laplacian similarly as in the case of circle patterns and 
conformally equivalent triangulations. Thus convex optimization may be 
applied for obtaining critical points. Convexity also shows that the solutions 
are unique (up to a similarity transformation). Therefore, in appropriate 
coordinates, the triangulations corresponding to discrete $\vartheta$-conformal 
maps are locally rigid. It remains an open question to 
determine an explicit formula for this functional.

Our considerations in this article are motivated by a class of 
examples for discrete 
$\vartheta$-conformal maps, namely, conformally symmetric triangular lattices,
see Section~\ref{SecEx}. 
These can be considered as discrete exponential functions or discrete Airy 
functions.
We characterize and study this class of examples in Section~\ref{SecConfSym}.

Note that there exist several other approaches for discrete conformal maps.
The earliest investigations 
resulted in the linear theory of discrete holomorphic maps, which also arises 
as suitable linearization of the (nonlinear) notions mentioned above, see for 
example~\cite{BMS05,BS08,Bue17}. Further studies of this approach can for 
example be found in~\cite{F,G16}. Apart from applications in 
numerics~\cite{C78,Sko13,DLS18}, this linear theory has been used for a 
rigorous study of dimers and the $2D$-Ising model in the context of 
probability, see~\cite{Ke02,S10,ChS11,ChS12}. 

\begin{figure}[tb]
\hfill
\begin{tikzpicture}[scale=0.7]
  \coordinate [label={left:$v_i$}] (A) at (0, 0);
  \coordinate [label={right:$v_j$}] (B) at (2cm,0);
  \coordinate [label={above:$v_k$}] (C) at (1cm, 1.732cm);
\coordinate [label={below:$v_l$}] (D) at (1cm, -1.732cm);

\draw [thick] (A) -- (D) -- (B) -- (C) -- (A);

   \tkzCircumCenter(A,B,C)\tkzGetPoint{G1}
    \tkzDrawCircle[dotted](G1,A)
  \tkzCircumCenter(A,B,D)\tkzGetPoint{G2}
    \tkzDrawCircle[dotted](G2,A)
  \draw [thick] (A) -- (B) ;
\end{tikzpicture}
\hfill
\begin{tikzpicture}[scale=0.7]
  \coordinate [label={left:$z_i$}] (A) at (0, 0);
  \coordinate [label={right:$z_j$}] (B) at (2cm,0);
  \coordinate [label={above:$z_k$}] (C) at (1.5cm, 1.732cm);
\coordinate [label={below:$z_l$}] (D) at (1.333cm, -0.81cm);

\draw [thick] (A) -- (D) -- (B) -- (C) -- (A);
  \draw [thick] (A) -- (B) ;
\end{tikzpicture}
\hfill
\begin{tikzpicture}[scale=0.7]
  \coordinate [label={left:$z_i$}] (A) at (0, 0);
  \coordinate [label={right:$z_j$}] (B) at (2.024cm,-0.022);
  \coordinate [label={above:$z_k$}] (C) at (1.697cm, 0.998cm);
\coordinate [label={below:$z_l$}] (D) at (-0.265cm, -1.199cm);

\draw [thick] (A) -- (D) -- (B) -- (C) -- (A); 

   \tkzCircumCenter(A,B,C)\tkzGetPoint{G1}
    \tkzDrawCircle[dotted](G1,A)
  \tkzCircumCenter(A,B,D)\tkzGetPoint{G2}
    \tkzDrawCircle[dotted](G2,A)
  \draw [thick] (A) -- (B) ;
\end{tikzpicture}
\hfill
\begin{tikzpicture}[scale=0.5]
  \coordinate [label={left:$z_i$}] (A) at (0, 0);
  \coordinate [label={right:$z_j$}] (B) at (2cm,0);
  \coordinate [label={above:$z_k$}] (C) at (2.175cm, 1.605cm);
\coordinate [label={below:$z_l$}] (D) at (0.131cm, -2.659cm);

\draw [thick] (A) -- (D) -- (B) -- (C) -- (A);

   \tkzCircumCenter(A,B,C)\tkzGetPoint{G1}
    \tkzDrawCircle[dotted](G1,A)
  \tkzCircumCenter(A,B,D)\tkzGetPoint{G2}
    \tkzDrawCircle[dotted](G2,A)
  \draw [thick] (A) -- (B) ;
\end{tikzpicture}
\hspace{2em}
\hfill
 \caption{Examples of discrete $\vartheta$-conformal functions for two 
incident triangles (from left to right): original configuration of two 
equilateral triangles, 
$\vartheta=0$ (discrete conformal equivalence), $\vartheta=\pi/2$ (circle 
pattern), $\vartheta=\pi/3$ (new notion)}\label{figExvartheta}
\end{figure}
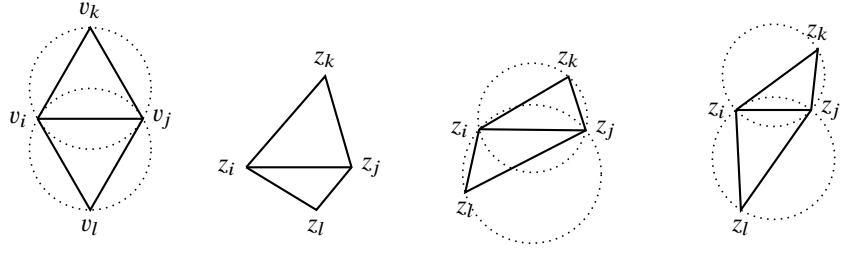

\section{Definitions and basic properties}\label{secDef}

In any discrete setting, the domain of a smooth conformal map is replaced 
by a suitable discrete object. In the following, let $T$ denote a (locally) 
embedded non-degenerate triangulation in the plane. Denote the vertices and 
edges of $T$ by $V$ and $E$ respectively.
 Edges will often be written as $e=[v_i,v_j]\in E$, where $v_i,v_j\in 
V$ are its incident vertices. For triangular faces we use the notation 
$\Delta[v_i,v_j,v_k]$ enumerating the incident vertices with respect to the
(counterclockwise) orientation of $\C$.

As a starting point for discrete conformal maps on $T$, we consider mappings 
$F:T\to\C$ which are continuous, piecewise affine-linear on every triangle, and 
orientation preserving. For every interior vertex, all its adjacent triangles build a {\em flower}. 
 $F$ is called a {\em discrete 
immersion} if for every flower all pairs of image triangles only intersect 
in their common vertex or edge, so the image flower is embedded. The 
immersed image triangulation will be denoted by $\widehat{T}=F(T)$ with 
vertices 
$\widehat{V}=F(V)$ and edges $\widehat{E}=F(E)$.
Given such a discrete immersion, when should it be called discrete conformal? 
In this article, we focus on the fact, that a smooth conformal map 
infinitesimally 
preserves cross-ratios. For the two triangulations $T$ and $\widehat{T}$ we can 
compute the 
cross-ratios of adjacent triangles in $T$ and $\widehat{T}$, respectively. We 
will use the 
following definition of the \emph{cross-ratio of four points} $z_1,z_2,z_3,z_4\in\C$: 
\begin{equation*}
\cro(z_1,z_2,z_3,z_4)= \frac{(z_1-z_2)}{(z_2-z_3)} \frac{(z_3-z_4)}{(z_4-z_1)}.
\end{equation*} 
For every interior edge $[v_i,v_j]$ in $T$ with adjacent triangles 
$\Delta[v_i,v_j,v_k]$ and $\Delta[v_i,v_l,v_j]$ as in 
Figure~\ref{Fig2Triang}, we define
\begin{align}
Q([v_i,v_j])&:=
\cro(v_i,v_l,v_j, v_k),\label{eqdefQ} \\
q([v_i,v_j])&:=
\cro(z_i,z_l,z_j,z_k),\label{eqdefq} 
\end{align}
where $z_m=F(v_m)$ for $m\in\{i,j,k,l\}$.
These cross-ratios are related to the geometric configuration of two incident 
triangles as
$q([v_i,v_j])=|q([v_i,v_j])|\text{e}^{i\varphi}$, where $|q([v_i,v_j])|= 
|\cro(z_i,z_l,z_j,z_k)|$ is also called {\em length cross-ratio} and 
$\varphi\in(0,\pi)$ is the interior intersection angle of the two circumcircles 
of the triangles $\Delta[z_i,z_l,z_j]$ and 
$\Delta[z_i,z_j,z_k]$ in $\widehat{T}$.
Note that for two embedded, non-degenerate, counterclockwise oriented triangles the logarithm of 
$q$ (and analogously $\log Q$) is well defined with values in $\R+i(0,\pi)$ by taking $\log 
(q([v_i,v_j])) := \log \frac{z_1-z_2}{z_3-z_2} +\log\frac{z_3-z_4}{z_1-z_4}$.
Furthermore, the values of $q$ characterize the image 
triangulation $\widehat{T}$ up to M\"obius transformations, see also
Lemma~\ref{theoq} below.
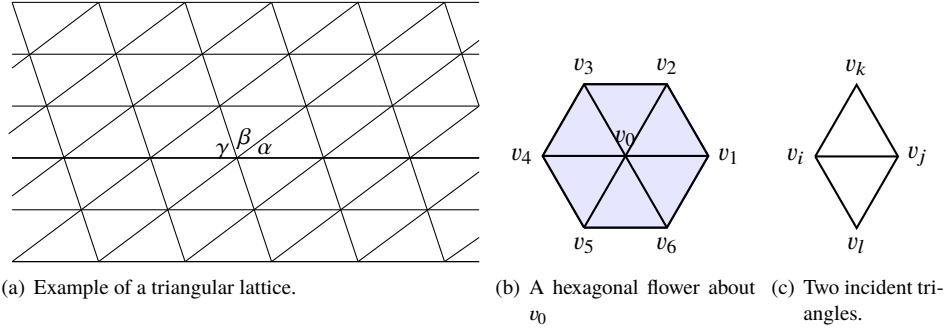
\begin{figure}[t]
\hfill
\subfigure[Example of a triangular lattice.]{\label{FigTiling}

\ifx\XFigwidth\undefined\dimen1=0pt\else\dimen1\XFigwidth\fi
\divide\dimen1 by 12264
\ifx\XFigheight\undefined\dimen3=0pt\else\dimen3\XFigheight\fi
\divide\dimen3 by 6774
\ifdim\dimen1=0pt\ifdim\dimen3=0pt\dimen1=4143sp\dimen3\dimen1
  \else\dimen1\dimen3\fi\else\ifdim\dimen3=0pt\dimen3\dimen1\fi\fi
\tikzpicture[x=+\dimen1, y=+\dimen3,scale=0.23]
\clip(-102,-8562) rectangle (12162,-1788);
\tikzset{inner sep=+0pt, outer sep=+0pt}
\draw (0,-8550)--(12150,-8550);
\draw (0,-7200)--(12150,-7200);
\draw (0,-4500)--(12150,-4500);
\draw (0,-3150)--(12150,-3150);
\draw (0,-1800)--(12150,-1800);
\draw (0,-8550)--(9000,-1800);
\draw (0,-8550)--(12150,-8550);
\draw (2250,-8550)--(0,-1800);
\draw (2250,-8550)--(11250,-1800);
\draw (6750,-8550)--(12150,-4500);
\draw (11250,-8550)--(9000,-1800);
\draw (9000,-8550)--(6750,-1800);
\draw (6750,-8550)--(4500,-1800);
\draw (4500,-8550)--(2250,-1800);
\draw (6750,-1800)--(0,-6840);
\draw (4500,-1800)--(-90,-5220);
\draw (2250,-1800)--(0,-3510);
\draw (11250,-1800)--(12150,-4500);
\draw (4500,-8550)--(12150,-2790);
\draw (9000,-8550)--(12150,-6210);
\draw (11250,-8550)--(12150,-7920);
\draw (0,-5850)--(12150,-5850);
\draw (0,-5850)--(12150,-5850);
\draw (0,-5850)--(12150,-5850);
\pgftext[base,left,at=\pgfqpointxy{6390}{-5760}] {\fontsize{40}{26.4}\normalfont $\alpha$};
\pgftext[base,left,at=\pgfqpointxy{5850}{-5480}] {\fontsize{40}{26.4}\normalfont $\beta$};
\pgftext[base,left,at=\pgfqpointxy{5300}{-5760}] {\fontsize{40}{26.4}\normalfont $\gamma$};
\endtikzpicture%
}
\hfill
\subfigure[A hexagonal flower about $v_0$]{\label{Figdefalpha}
\begin{tikzpicture}[thick, scale=0.55]
 
  \coordinate [label={right:$v_1$}] (B) at (2cm,0);
  \coordinate [label={above:$v_2$}] (C) at (1cm, 1.732cm);
\coordinate [label={below:$v_6$}] (D) at (1cm, -1.732cm);
\coordinate [label={left:$v_4$}] (E) at (-2cm,0);
  \coordinate [label={above:$v_3$}] (F) at (-1cm, 1.732cm);
\coordinate [label={below:$v_5$}] (G) at (-1cm, -1.732cm);

\draw [thick] (G) -- (D) -- (B) -- (C) -- (F) -- (E) -- (G); 
\filldraw[fill=blue!10!white] (G) -- (D) -- (B) -- (C) -- (F) -- (E) -- (G);
 \coordinate [label={above:$v_0$}] (A) at (0, 0);
  \draw [thick] (A) -- (B) ;
\draw [thick] (A) -- (C) ;
\draw [thick] (A) -- (D) ;
\draw [thick] (A) -- (E) ;
\draw [thick] (A) -- (F) ;
\draw [thick] (A) -- (G) ;
\end{tikzpicture}
}
\hfill
\subfigure[Two incident triangles.]{\label{Fig2Triang}
\begin{tikzpicture}[scale=0.55]
  \coordinate [label={left:$v_i$}] (A) at (0, 0);
  \coordinate [label={right:$v_j$}] (B) at (2cm,0);
  \coordinate [label={above:$v_k$}] (C) at (1cm, 1.732cm);
\coordinate [label={below:$v_l$}] (D) at (1cm, -1.732cm);

\draw [thick] (A) -- (D) -- (B) -- (C) -- (A);
  \draw [thick] (A) -- (B) ;
\end{tikzpicture}
}
\hspace{2em}
\hfill
\caption{Lattice triangulation $TL_\C$ of the plane with congruent 
triangles, a flower of $TL_\C$ about $v_0$, and two incident 
triangles.}\label{FigRegTile}
\end{figure}

Similarly as for smooth conformal maps, we could demand that $F$ preserves all 
cross-ratios, so $q\equiv Q$. But then $\widehat{T}$ is only the image of $T$ by 
a M\"obius transformation. Therefore, it seems reasonable that only 
``half'' of the cross-ratios remains unchanged, in particular:
\begin{definition}\label{discconfdef}
 The discrete immersion $F:T\to\C$ with image triangulation 
$\widehat{T}$ 
is called {\em discrete $\vartheta$-conformal} for a constant 
$\vartheta\in[0,\pi/2]$ if for all interior edges $[v_i,v_j]$ with 
adjacent triangles 
$\Delta[v_i,v_l,v_j]$ and $\Delta[v_i,v_j,v_k]$ there holds
\begin{equation}\label{eqdefdiscconf}
\text{Re}[\text{e}^{-i\vartheta}\log(Q([v_i,v_j]))]= 
\text{Re}[\text{e}^{-i\vartheta}\log(q([v_i,v_j]))],
\end{equation}
where we assume the values of the logarithm to be in $\R+i(0,2\pi)$.
\end{definition}

Note that this definition contains
 two known notions of discrete conformality based on circle patterns and 
discrete conformal equivalence.
\begin{itemize}
 \item For $\vartheta=0$
the absolute values of the cross-ratios (length cross-ratios) for 
adjacent triangles agree, that is $|Q([v_i,v_j])|=|q([v_i,v_j])|$ 
holds for all interior edges $[v_i,v_j]$. Therefore, the triangulations $T$ and 
$\widehat{T}$ are {\em discretely conformally 
equivalent}, see for example~\cite[Prop.~2.3.2]{BPS13}. 
\item For $\vartheta=\pi/2$ the arguments of the 
cross-ratios for adjacent triangles agree, that is $\arg Q([v_i,v_j])= \arg q([v_i,v_j])$. These 
equal the intersection angles 
of the corresponding circumcircles. Adding these circumcircles
for all triangles in $T$ and in $\widehat{T}$, we obtain a discrete conformal 
map between two {\em circle patterns}.
\end{itemize}
Further properties of discrete $\vartheta$-conformal maps will be studied in 
Sections~\ref{SecExT} and~\ref{SecVari}.

Our considerations in this article are motivated by the study of a class of 
examples, where the cross-ratios defined in~\eqref{eqdefq} are 
locally constant. 
To this end, we restrict ourselves to the 
case where the triangulation $T$ is a (part of a) {\em triangular lattice} 
$TL_{\C}$, that is, a lattice triangulation of the whole complex plane $\C$ 
with congruent triangles, see Figure~\ref{FigTiling}. We denote the 
triangulation $TL=T$ in this case in order to emphasize the lattice structure. 
For some considerations we need the actual geometric data of the lattice 
$TL_{\C}$, i.e.\ the angles 
$\alpha,\beta,\gamma\in (0,\pi)$ such that $\alpha+\beta+\gamma=\pi$, which is
specified according to the notation in Figure~\ref{FigTiling}.
Obviously, the cross-ratios $Q$ are constant on parallel edges (as $TL$ is part 
of a lattice). 

We start with a useful property of the cross-ratio function.
To this end, we introduce some notation.
A vertex $v_0\in V$ is called {\em interior vertex} if it is an interior 
point of the union of its adjacent closed triangular faces, see 
Figure~\ref{Figdefalpha}.
We call an interior vertex $v_0$ of ${T}$ together 
with its incident triangles (including their vertices and edges) a {\em 
flower}. The vertex $v_0$ is called the {\em center of the flower} and its 
neighbors will 
frequently be enumerated corresponding to the cyclic order of the 
corresponding vertices of $T$.

\begin{lemma}\label{theoq}
Let $TL$ be a part of a triangular lattice.
If a function $q:EL_{int}\to\C\setminus\R_{\geq 0}$ originates from a 
discrete immersion of $TL$ then
for every interior vertex $v_0\in VL$ with cyclic 
neighbors $v_1,\dots, v_6\in V$ and $q_k=q([v_0,v_k])$ there holds
\begin{align}
 &\sum_{k=1}^6\arg(q_k)=4\pi\quad \Rightarrow\quad q_1q_2q_3q_4q_5q_6=1, 
\label{eqq1}\\
&1-q_1 +q_1q_2 -q_1q_2q_3 +q_1q_2q_3q_4 -q_1q_2q_3q_4q_5=0.\label{eqq2}
\end{align}
\end{lemma}
\begin{proof}
Given an embedded flower of $\widehat{TL}$ about $z_0$, we
add circumcircles to the triangles and then apply a M\"obius transformation to 
this configuration which maps $z_0$ to $\infty$ as indicated in 
Figure~\ref{figFlowerMoeb}.
As the M\"obius transformation does not change the values of the corresponding 
cross-ratios $q_k$, we easily 
identify equations~\eqref{eqq1} and~\eqref{eqq2} as the closing conditions for 
the image polygon. 
\end{proof}
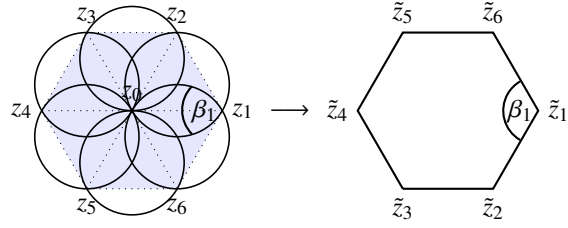
\begin{figure}
\begin{center}
\begin{tikzpicture}[scale=0.6]
  \coordinate [label={right:$z_1$}] (B) at (2cm,0);
  \coordinate [label={above:$z_2$}] (C) at (1cm, 1.732cm);
\coordinate [label={below:$z_6$}] (D) at (1cm, -1.732cm);
\coordinate [label={left:$z_4$}] (E) at (-2cm,0);
  \coordinate [label={above:$z_3$}] (F) at (-1cm, 1.732cm);
\coordinate [label={below:$z_5$}] (G) at (-1cm, -1.732cm);

\filldraw[dotted, fill=blue!10!white] (G) -- (D) -- (B) -- (C) -- (F) -- (E) -- 
(G);
 \coordinate [label={above:$z_0$}] (A) at (0, 0);
  \draw [dotted] (A) -- (B) ;
\draw  [dotted] (A) -- (C) ;
\draw [dotted] (A) -- (D) ;
\draw [dotted] (A) -- (E) ;
\draw [dotted] (A) -- (F) ;
\draw [dotted] (A) -- (G) ;

    \tkzCircumCenter(A,B,C)\tkzGetPoint{G1}
    \tkzDrawCircle(G1,A)
  \tkzCircumCenter(A,B,D)\tkzGetPoint{G2}
    \tkzDrawCircle(G2,A)
  \tkzCircumCenter(A,C,F)\tkzGetPoint{G3}
    \tkzDrawCircle(G3,A)
  \tkzCircumCenter(A,F,E)\tkzGetPoint{G4}
    \tkzDrawCircle(G4,A)
  \tkzCircumCenter(A,E,G)\tkzGetPoint{G5}
    \tkzDrawCircle(G5,A)
  \tkzCircumCenter(A,G,D)\tkzGetPoint{G6}
    \tkzDrawCircle(G6,A)

\draw (1.6cm,0) node {$\beta_1$};
\draw[thick]  (1.33cm,-0.52cm) arc (225:135:0.75cm);

\draw (3.5cm,0) node {$\longrightarrow$};

 \coordinate [label={right:$\tilde{z}_1$}] (B1) at (9cm,0);
  \coordinate [label={above:$\tilde{z}_6$}] (C1) at (8cm, 1.732cm);
\coordinate [label={below:$\tilde{z}_2$}] (D1) at (8cm, -1.732cm);
\coordinate [label={left:$\tilde{z}_4$}] (E1) at (5cm,0);
  \coordinate [label={above:$\tilde{z}_5$}] (F1) at (6cm, 1.732cm);
\coordinate [label={below:$\tilde{z}_3$}] (G1) at (6cm, -1.732cm);
\draw [thick] (G1) -- (D1) -- (B1) -- (C1) -- (F1) -- (E1) -- (G1);

\draw (8.6cm,0) node {$\beta_1$};
\draw[thick]  (8.6cm,-0.65cm) arc (240:120:0.75cm);
\end{tikzpicture}
\end{center}
\caption{Mapping the circumcircles of a flower by a M\"obius 
transformation with $z_0\mapsto\infty$}\label{figFlowerMoeb}
\end{figure}

\begin{remark}\label{remGen}
 Lemma~\ref{theoq} is formulated for flowers consisting of six triangles 
 because we are particularly interested in this combinatorics for 
 conformally symmetric triangular lattices. But 
equations~\eqref{eqq1} and~\eqref{eqq2} and their proof can easily 
be generalized for flowers in arbitrary immersed triangulations ${T}$.
\end{remark}

\begin{remark}\label{remq}
Given a cross-ratio function $q$ on the edges of a flower of $\widehat{TL}$ 
incident to $v_0$, 
note that equations~\eqref{eqq1} and~\eqref{eqq2} are not sufficient to define 
a discrete embedding of this flower because the interiors of different 
triangles may intersect. Nevertheless, from the values of $q$ we can 
always build a corresponding hexagon as in Figure~\ref{figFlowerMoeb} (right). 
But this hexagon is possibly not embedded.
Applying a M\"obius transformation which maps $\infty$ to a finite point and 
maps none of the vertices of the hexagon to $\infty$, we obtain a configuration 
of six triangles. These are the images of a piecewise linear map of a flower of 
$TL_\C$. By construction, this configuration is unique up to M\"obius 
transformations. In case that the resulting configuration is an embedded 
flower, we obtain a discrete immersion of this flower. Given further values of 
the cross-ratio $q$ on the boundary edges of the flower, these determine new 
vertices and triangles. If all these triangles are locally embedded about 
every flower, 
we can continue this procedure and finally obtain a discrete immersion from the values 
of the cross-ratio $q$ on the edges.
\end{remark}

\section{Conformally symmetric triangular lattices}\label{SecConfSym}

In this section we study conformally symmetric discrete immersions and exploit 
ideas used in~\cite{BH01} for conformally symmetric circle packings. We choose 
the cross-ratios of incident triangles as suitable 
M\"obius invariant parameters. Note that throughout this section we will
work with (parts of) triangular lattices $TL$ instead of a general 
triangulation~$T$.

The following lemma gives examples of cross-ratio functions as considered in 
Lemma~\ref{theoq} which always lead to discrete immersions.

\begin{lemma}\label{lemDoyle}
Let $q_1,q_2,q_3\in\Ha =\{z\in\C:\text{Im}(z)>0\}$ be three numbers such that 
$\arg(q_k)\in(0,\pi)$ and $\arg(q_1)+\arg(q_2)+\arg(q_3)=2\pi$. Define a 
function $q:E\to \{q_1,q_2,q_3\}$ on the edges of a triangular lattice 
such that $q$ assumes the same value on parallel edges of $TL$. Then there exists 
a discrete immersion $\widehat{TL}$, which will be called {\em 
generalized Doyle spiral}, such that $q$ is the corresponding cross-ratio 
function.
\end{lemma}
\begin{proof}
By construction, $q$ satisfies the 
necessary conditions~\eqref{eqq1} and~\eqref{eqq2} for every flower.
Furthermore, the hexagon built corresponding to the values of $q$ on a flower 
is symmetric and also embedded as $q_1,q_2,q_3\in\Ha$. Therefore, there exists 
a M\"obius transformation which maps this hexagon to an embedded flower. In 
particular, we can take an inversion which maps $\infty$ to the intersection 
point of the diagonals of the symmetric hexagon. Then the resulting flower is 
still symmetric in the following sense: for pair of incident triangles there 
exists a similarity transformation which maps this pair onto the opposite pair 
of incident triangles (for example $\Delta[z_0,z_1,z_2]\cup\Delta[z_0,z_2,z_3]$ 
onto $\Delta[z_0,z_4,z_5]\cup\Delta[z_0,z_5,z_6]$ in the notation of 
Figure~\ref{figFlowerMoeb}). Therefore, if we start with such a flower we can 
continue and use the values of $q$ to determine further images of vertices as 
indicated in Remark~\ref{remq}. By symmetry, all flowers will be similar and 
embedded.
\end{proof}

Doyle spirals carry a 
lot of symmetry which can be expressed by the fact, that 
the cross-ratios are constant for all parallel edges of $TL$. This is 
in fact a special case of
a {\em conformally symmetric triangular lattice} $\widehat{TL}$ which only 
contains conformally symmetric flowers defined as follows.
 \begin{definition}\label{DefM}
  A flower of $\widehat{TL}$ with center $y_0\in \widehat{VL}$ and incident
vertices $z_1,\dots,z_6\in \widehat{VL}$ 
in cyclic order is called {\em conformally symmetric} if there is an involutive 
M\"obius transformation $M$ with fixed point $y_0$ and such that
\begin{equation}
 M(z_k)=z_{k+3}\ \text{ for } k\ (\bmod 6).
\end{equation}
 \end{definition}

The M\"obius invariant notion of conformal symmetry can locally be  
characterized as follows.

\begin{theorem}\label{theoconfsym}
 For any flower in $\widehat{TL}$ about $z_0\in\widehat{V}$ 
with neighbors $z_1,\dots, z_6\in\widehat{V}$ in cyclic order the following 
statements are  equivalent.
\begin{enumerate}[(i)]
 \item The flower is conformally symmetric.
\item  Opposite cross-ratios agree, that is $q([v_0,v_k])=q([v_0,v_{k+3}])$ for 
$k=1,2,3$.
\item The circles/lines $C_k=C_k(z_k,z_0,z_{k+3})$, 
$k=1,2,3$, through $z_k,z_0,z_{k+3}$ have a second common intersection point 
$X\in\C\cup\{\infty\}$, apart from $z_0$, which satisfies 
$\cro(z_k,z_0,z_{k+3},X)=-1$.
\end{enumerate}
\end{theorem}
The proof follows by applying suitable M\"obius transformations as in the 
definition of conformal symmetry and is left to the reader.

Note that the statements of this theorem also hold if we replace $z_0$ by $X$ 
for the flower in consideration. But the corresponding flower about $X$ is not 
embedded and thus not part of the image of a discrete immersion $\widehat{TL}$.

\subsection{Generalized Doyle spirals}
The study of Doyle spirals began with an observation of Doyle for 
the construction of circle packing, see ~\cite{BDS94} for more details. 
We briefly recall the geometric construction of generalized Doyle spirals which 
is based on successive gluing of rescaled copies of a given convex 
quadrilateral.

For a given non-degenerate convex quadrilateral ${\mathcal Q}$ in $\C$ with 
vertices $A,B,C,D$ as in Figure~\ref{FigQuad}~(left),
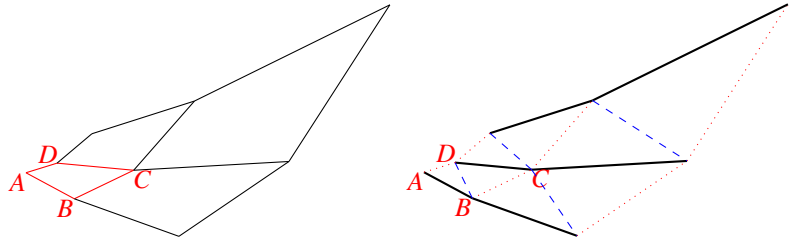
\begin{figure}
\begin{center}
\begin{tikzpicture} [rotate=-45,scale=0.7]
\draw [red] (1,0) -- (2,0.3);
\draw [red, rotate=23, scale=1.4] (1,0) -- (2,0.3);
\draw [red] (1,0) -- ({1.4*cos(23)}, {1.4*sin(23)});
\draw [red] (2,0.3) -- ({
1.4*cos(23)*2.022*cos(8.531) -1.4*sin(23)*2.022*sin(8.531)}, 
{1.4*sin(23)*2.022*cos(8.531) +1.4*cos(23)*2.022*sin(8.531)});

\draw [rotate=23*2, scale=1.4^2] (1,0) -- (2,0.3);
\draw [rotate=23, scale=1.4] (2,0.3) -- ({
1.4*cos(23)*2.022*cos(8.531) -1.4*sin(23)*2.022*sin(8.531)}, 
{1.4*sin(23)*2.022*cos(8.531) +1.4*cos(23)*2.022*sin(8.531)});
\draw [rotate=23, scale=1.4] (1,0) -- ({1.4*cos(23)}, {1.4*sin(23)});

\foreach \i in {0,...,1}{
\draw [rotate=23*\i, scale=1.4^\i] 
({2.022^2*cos(8.531*2)},{2.022^2*sin(8.531*2)})--({
1.4*cos(23)*2.022^2*cos(8.531*2)-1.4*sin(23)*2.022^2*sin(8.531*2)}, 
{1.4*sin(23)*2.022^2*cos(8.531*2)+1.4*cos(23)*2.022^2*sin(8.531*2)});
\draw [rotate=23*\i, scale=1.4^\i] 
({2.022*cos(8.531)},{2.022*sin(8.531)}) -- 
({2*2.022*cos(8.531)-0.3*2.022*sin(8.531)}, 
{0.3*2.022*cos(8.531)+2*2.022*sin(8.531)});
}
\draw [rotate=23*2, scale=1.4^2] 
({2.022*cos(8.531)},{2.022*sin(8.531)}) -- 
({2*2.022*cos(8.531)-0.3*2.022*sin(8.531)}, 
{0.3*2.022*cos(8.531)+2*2.022*sin(8.531)});

\draw (1,-0.25) node[red] {$A$};
\draw (2,0.05) node[red] {$B$};
\draw ({1.4*cos(23)-0.25}, {1.4*sin(23)}) node[red] {$D$};
\draw ({
1.4*cos(23)*2.022*cos(8.531) -1.4*sin(23)*2.022*sin(8.531) +0.25}, 
{1.4*sin(23)*2.022*cos(8.531) +1.4*cos(23)*2.022*sin(8.531)}) node[red] {$C$};
\end{tikzpicture}
\begin{tikzpicture}[rotate=-45,scale=0.7]
\draw [thick] (1,0) -- (2,0.3);
\draw [thick, rotate=23, scale=1.4] (1,0) -- (2,0.3);
\draw [dotted, red] (1,0) -- ({1.4*cos(23)}, {1.4*sin(23)});
\draw [dotted, red] (2,0.3) -- ({
1.4*cos(23)*2.022*cos(8.531) -1.4*sin(23)*2.022*sin(8.531)}, 
{1.4*sin(23)*2.022*cos(8.531) +1.4*cos(23)*2.022*sin(8.531)});

\draw [thick, rotate=23*2, scale=1.4^2] (1,0) -- (2,0.3);
\draw [dotted, red, rotate=23, scale=1.4] (2,0.3) -- ({
1.4*cos(23)*2.022*cos(8.531) -1.4*sin(23)*2.022*sin(8.531)}, 
{1.4*sin(23)*2.022*cos(8.531) +1.4*cos(23)*2.022*sin(8.531)});
\draw [dotted, red, rotate=23, scale=1.4] (1,0) -- ({1.4*cos(23)}, 
{1.4*sin(23)});

\foreach \i in {0,...,1}{
\draw [dotted, red, rotate=23*\i, scale=1.4^\i] 
({2.022^2*cos(8.531*2)},{2.022^2*sin(8.531*2)})--({
1.4*cos(23)*2.022^2*cos(8.531*2)-1.4*sin(23)*2.022^2*sin(8.531*2)}, 
{1.4*sin(23)*2.022^2*cos(8.531*2)+1.4*cos(23)*2.022^2*sin(8.531*2)});
\draw [thick, rotate=23*\i, scale=1.4^\i] 
({2.022*cos(8.531)},{2.022*sin(8.531)}) -- 
({2*2.022*cos(8.531)-0.3*2.022*sin(8.531)}, 
{0.3*2.022*cos(8.531)+2*2.022*sin(8.531)});
}
\draw [thick, rotate=23*2, scale=1.4^2] 
({2.022*cos(8.531)},{2.022*sin(8.531)}) -- 
({2*2.022*cos(8.531)-0.3*2.022*sin(8.531)}, 
{0.3*2.022*cos(8.531)+2*2.022*sin(8.531)});

\draw (1,-0.25) node[red] {$A$};
\draw (2,0.05) node[red] {$B$};
\draw ({1.4*cos(23)-0.25}, {1.4*sin(23)}) node[red] {$D$};
\draw ({
1.4*cos(23)*2.022*cos(8.531) -1.4*sin(23)*2.022*sin(8.531) +0.25}, 
{1.4*sin(23)*2.022*cos(8.531) +1.4*cos(23)*2.022*sin(8.531)}) node[red] {$C$};

\draw [blue, dashed] ({1.4*cos(23)}, {1.4*sin(23)}) -- (2,0.3);
\draw [blue, dashed] ({
1.4*cos(23)*2.022*cos(8.531) -1.4*sin(23)*2.022*sin(8.531)}, 
{1.4*sin(23)*2.022*cos(8.531) +1.4*cos(23)*2.022*sin(8.531)}) -- 
({2*2.022*cos(8.531)-0.3*2.022*sin(8.531)}, 
{0.3*2.022*cos(8.531)+2*2.022*sin(8.531)});
\draw [blue, dashed, rotate=23, scale=1.4] ({1.4*cos(23)}, {1.4*sin(23)}) -- 
(2,0.3);
\draw [blue, dashed, rotate=23, scale=1.4] ({
1.4*cos(23)*2.022*cos(8.531) -1.4*sin(23)*2.022*sin(8.531)}, 
{1.4*sin(23)*2.022*cos(8.531) +1.4*cos(23)*2.022*sin(8.531)}) -- 
({2*2.022*cos(8.531)-0.3*2.022*sin(8.531)}, 
{0.3*2.022*cos(8.531)+2*2.022*sin(8.531)});
\end{tikzpicture}
\end{center}
\caption{Geometric construction of a generalized Doyle spiral 
and splitting into triangles.}\label{FigQuad}
\end{figure}
let $L_1,L_2$ be the orientation preserving similarity transformations such 
that $L_1(A)=B$, $L_1(D)=C$, $L_2(A)=D$, $L_2(B)=C$.
Then
$L_2(L_1({\mathcal Q})) =L_1(L_2({\mathcal Q}))$ by angle count, so 
$L_2L_1=L_1L_2$.
Therefore $L_1,L_2$ are both translations or both scale-rotations with the same 
center of rotation.
Thus, successive applications of $L_1,L_2$ generate either a lattice or a 
generalized Doyle spiral spiraling  about the common center, see 
Figure~\ref{FigExDoyle} for an example.

\begin{figure}[t]
\begin{center}
\begin{tikzpicture} [scale=0.2]
\clip(-25,-10) rectangle (25,20);
\foreach \j in {0,...,6}
\foreach \i in {-10,...,7}{
\draw [rotate=23*\i, scale=1.4^\i] 
({2.022^\j*cos(8.531*\j)},{2.022^\j*sin(8.531*\j)}) -- 
({2*2.022^\j*cos(8.531*\j)-0.3*2.022^\j*sin(8.531*\j)}, 
{0.3*2.022^\j*cos(8.531*\j)+2*2.022^\j*sin(8.531*\j)});
\draw [rotate=23*\i, scale=1.4^\i] 
({2.022^\j*cos(8.531*\j)},{2.022^\j*sin(8.531*\j)})--({
1.4*cos(23)*2.022^\j*cos(8.531*\j)-1.4*sin(23)*2.022^\j*sin(8.531*\j)}, 
{1.4*sin(23)*2.022^\j*cos(8.531*\j)+1.4*cos(23)*2.022^\j*sin(8.531*\j)});
}
\foreach \i in {-10,...,0}{
\draw [rotate=23*\i, scale=1.4^\i] 
({2.022^7*cos(8.531*7)},{2.022^7*sin(8.531*7)})--({
1.4*cos(23)*2.022^7*cos(8.531*7)-1.4*sin(23)*2.022^7*sin(8.531*7)}, 
{1.4*sin(23)*2.022^7*cos(8.531*7)+1.4*cos(23)*2.022^7*sin(8.531*7)});
}
\end{tikzpicture}
\end{center}
\caption{Example of a part of a generalized Doyle spiral from a 
quad.}\label{FigExDoyle}
\end{figure}
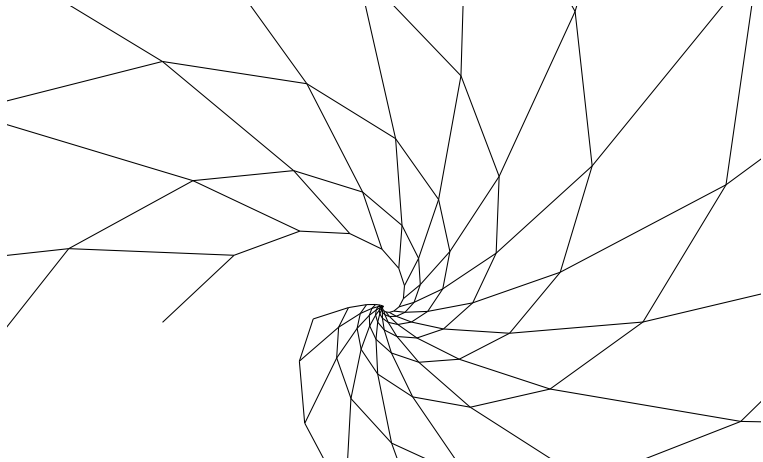

If we split the quadrilateral ${\mathcal Q}$ and all its images consistently into 
two triangles as in Figure~\ref{FigQuad} (right), we 
obtain a triangulation $\widehat{TL}$ which is the image of a discrete 
immersion of the whole triangular lattice $TL_\C$. It is 
easy to see that the cross-ratios $q$ defined in~\eqref{eqdefq} are constant on 
each of the three types of edges and these three values multiply to~$1$.

 \begin{remark}
A triangulation
$\widehat{T}$ is a (part of a) generalized Doyle spiral if and only if 
$\widehat{T}$ is invariant under the M\"obius transformation $\mathcal M$ given in 
Definition~\ref{DefM}. This is also equivalent to the fact that 
all circles $C_k$ considered in part~(iii) of Theorem~\ref{theoconfsym} 
intersect in the same point (the center of the spiral) for all flowers of 
$\widehat{T}$.
\end{remark}

\begin{remark}[Connection to Doyle spirals for circle packings]
 Start with a quadrilateral $\mathcal Q$ constructed from the midpoints of four 
mutually tangent circles (this is possible if and only if $a+c=b+d$ holds for 
the edge lengths), see Figure~\ref{FigDoylePack}. Furthermore, assume that 
$L_1$ and $L_2$ map 
circles onto circles (which is equivalent to the condition $r_1r_3=r_2r_4$ for 
the radii). Then successive applications of $L_1$ and $L_2$ to the first four 
circles will result in a hexagonal circle packing called {\em Doyle spiral}, 
see~\cite{BDS94}.
\end{remark}

 \begin{figure}
\begin{center}
 \begin{tikzpicture}[scale=0.8]
   \coordinate [label={left:$A$}] (A) at (0,0);
  \coordinate [label={right:$B$}] (B) at (2cm, 0);
\coordinate [label={left:$D$}] (D) at (1cm,1.732cm);
\coordinate [label={right:$C$}] (C) at (3cm,1.732cm);
 
  \draw [thick] (A) --  (B) node[midway,below] {$a$} ;
\draw  [thick] (A) -- (D) node[near start,above] {$d$} ;
\draw [thick] (B) --(C) node[near end,below] {$b$};
\draw [thick] (C) -- (D) node[midway,below] {$c$};
 \end{tikzpicture}
 \hspace{3em}
 \begin{tikzpicture}[scale=0.6]
   \coordinate [label={left:$A$}] (A) at (0,0);
  \coordinate [label={right:$B$}] (B) at (2cm, 0);
\coordinate [label={left:$D$}] (D) at (1cm,1.732cm);
\coordinate [label={right:$C$}] (C) at (3cm,1.732cm);
 
  \draw [thick] (A) -- (B) node[near start,below] {$r_1$} node[near end,below] 
{$r_2$};
\draw  [thick] (A) -- (D) ;
\draw [thick] (B) --(C) ;
\draw [thick] (C) -- (D) node[near start,below] {$r_3$} node[near end,below] 
{$r_4$};

\draw (A) circle (1cm);
\draw (B) circle (1cm);
\draw (C) circle (1cm);
\draw (D) circle (1cm);
 \end{tikzpicture}
\end{center}
\caption{Construction of Doyle spirals for circle packings 
from four mutually tangent circles.}\label{FigDoylePack}
\end{figure}
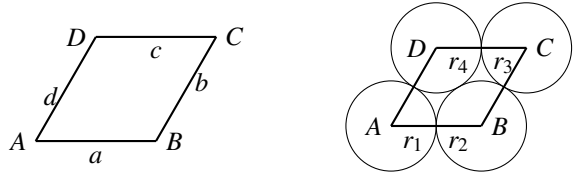

\begin{remark}[Generalized Doyle spirals for conformally equivalent 
tri\-an\-gu\-la\-tions]
\label{remlog}
For two conformally equivalent triangulations $TL$ and $\widehat{TL}=F(TL)$
there exists a function $u:VL\to\R$, called {\em logarithmic 
scale factors}, such that the edge lengths are related by ${\widehat l}([v,w])= 
l([vw]) \text{e}^{(u(v)+u(w))/2}$ 
for all $[v,w]\in EL$,  see for example~\cite{BPS13}. The 
logarithmic scale factors for corresponding generalized Doyle spirals on 
$TL_\C$ with 
$VL_\C=\{n\sin\alpha_2+ 
m\sin\alpha_3\text{e}^{i\alpha_1}: n,m\in\Z\}$ are linear,
\[u(n\sin\alpha_2+ m\sin\alpha_3\text{e}^{i\alpha_1})= An+Bm
  \qquad \text{for suitable } A,B\in\R,\]  
which indicates that this discrete conformal map can be considered 
as an analogue of the exponential map. See also~\cite[Lemma~2.6]{WGS15} 
and~\cite{Bue17}.
 \end{remark}

\subsection{Parametrization of conformally symmetric triangular lattices}
Consider any function $q:EL_{int}\to\C$ which satisfies condition~(ii) of 
Theorem~\ref{theoconfsym}, namely opposite cross-ratios agree,
\begin{equation}\label{eqconfsym}
q([v_0,v_k])=q([v_0,v_{k+3}]),\quad k=1,2,3,
\end{equation} 
for any flower in $T$, where $v_0$ denotes the center with neighbors 
$v_1,\dots, v_6$ in cyclic order. Then $q$ automatically satisfies~\eqref{eqq2}.
Combining~\eqref{eqconfsym} with~\eqref{eqq1}, we deduce that
\begin{equation}\label{eqprod}
 q([v_0,v_k])q([v_0,v_{k+1}])q([v_0,v_{k+2}])=1
\end{equation}
for any three consecutive edges in the flower about $v_0$. Thus we can describe 
the general solution of equation~\eqref{eqq2} together with~\eqref{eqconfsym} 
and~\eqref{eqprod} as three parameter family on the honeycomb 
lattice which is dual to the triangular lattice $TL$. To this end, consider the 
function $q$ (by abuse of notation) to be defined on the corresponding {\em 
dual} edges $e^*\in EL^*$ as in Figure~\ref{FigPara}. 
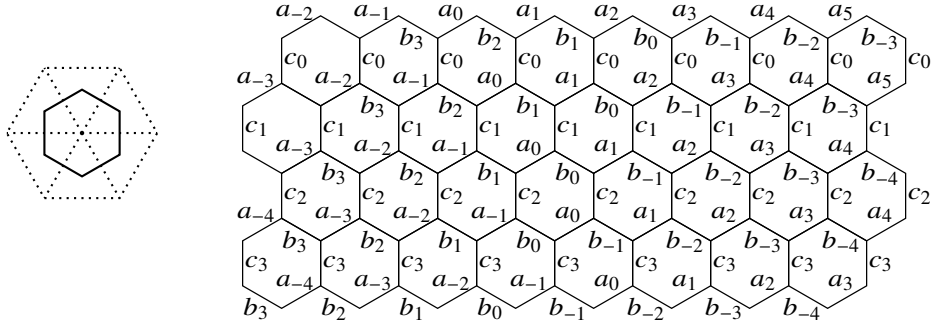
\begin{figure}
\raisebox{1.\height}{
\begin{tikzpicture}[thick, scale=0.5]
  \coordinate  (B) at (2cm,0);
  \coordinate  (C) at (1cm, 1.732cm);
\coordinate  (D) at (1cm, -1.732cm);
\coordinate  (E) at (-2cm,0);
  \coordinate  (F) at (-1cm, 1.732cm);
\coordinate  (G) at (-1cm, -1.732cm);

\draw[dotted] (G) -- (D) -- (B) -- (C) -- (F) -- (E) -- 
(G);
 \coordinate  (A) at (0, 0);
  \draw [dotted] (A) -- (B) ;
\draw [dotted] (A) -- (C) ;
\draw [dotted] (A) -- (D) ;
\draw [dotted] (A) -- (E) ;
\draw [dotted] (A) -- (F) ;
\draw [dotted] (A) -- (G) ;
\draw [thick] (1cm, 0.577cm) -- (0, 1.154cm) --(-1cm, 0.577cm) -- (-1cm, 
-0.577cm) -- (0, -1.154cm)--(1cm, -0.577cm) -- cycle;
\end{tikzpicture}
}
\hfill
\resizebox{0.77\textwidth}{!}{
\begin{tikzpicture}
  \tikzset{
    box/.style={
      regular polygon,
      regular polygon sides=6,
      minimum size=10mm,
      inner sep=0mm,
      outer sep=0mm,
      rotate=90,
    draw
    }
  }

\foreach \i in {0,...,7} 
    \foreach \j in {0,...,1} {
            \node[box] at (0.866*\i,1.5*\j) {};
            \node[box] at (0.866*\i+0.433,1.5*\j+0.75) {};
        }

\foreach \i in {0,...,8} {
\draw (0.866*\i-0.433-0.1,0) node[right] {\footnotesize $c_3$};
\draw (0.866*\i-0.1,0.75) node[right] {\footnotesize $c_2$};
\draw (0.866*\i-0.433-0.1,1.5) node[right] {\footnotesize $c_1$};
\draw (0.866*\i-0.1,1.5+0.75) node[right] {\footnotesize $c_{0}$};
}
\foreach \l in {-3,...,3}
\foreach \j in {-1,...,3} {
\draw (3*0.866-0.866*\l +0.15+0.75*0.577*\j,0.75*\j) node[above] 
{\footnotesize $b_{\scriptstyle\l}$};
}
\foreach \j in {-1,...,1} {
\draw (2*0.866+0.866+0.866*4 +0.15+0.75*0.577*\j,0.75*\j) node[above] 
{\footnotesize $b_{-4}$};
}
\foreach \k in {-4,...,3} 
\foreach \j in {0,...,0} {
\draw (4*0.866+0.866*\k-0.866*\j +0.15+0.75*0.577*\j,0.75*\j) node[below] 
{\footnotesize $a_{\scriptstyle\k}$};
}
\foreach \k in {-4,...,4} 
\foreach \j in {1,...,1} {
\draw (4*0.866+0.866*\k-0.866*\j +0.15+0.75*0.577*\j,0.75*\j) node[below] 
{\footnotesize $a_{\scriptstyle\k}$};
}
\foreach \k in {-3,...,4} 
\foreach \j in {2,...,2} {
\draw (4*0.866+0.866*\k-0.866*\j +0.15+0.75*0.577*\j,0.75*\j) node[below] 
{\footnotesize $a_{\scriptstyle\k}$};
}
\foreach \k in {-3,...,5} 
\foreach \j in {3,...,3} {
\draw (4*0.866+0.866*\k-0.866*\j +0.15+0.75*0.577*\j,0.75*\j) node[below] 
{\footnotesize $a_{\scriptstyle\k}$};
}
\foreach \k in {-2,...,5} 
\foreach \j in {4,...,4} {
\draw (4*0.866+0.866*\k-0.866*\j +0.15+0.75*0.577*\j,0.75*\j) node[below] 
{\footnotesize $a_{\scriptstyle\k}$};
}
\end{tikzpicture}
}
\caption{{\it Left:} The edges of a flower (dotted) and the hexagon of their
dual edges (black).
{\it Right:} Parametrization of conformally 
symmetric triangular lattices.}\label{FigPara}
\end{figure}

\begin{theorem}\label{theoabc}
 Let $\widehat{TL}$ be a conformally symmetric triangulation. Define the 
cross-ratio 
function $q:EL_{int}\to\C$ on the edges and consider it as a function on the 
dual edges $q:EL_{int}^*\to\C$. Consider an interior dual vertex and denote the 
corresponding cross-ratios on the incident edges by 
$a,b,c\in\C\setminus\R_{\geq 0}$. Then the dual graph $\widehat{TL}^*$ is part 
of a honeycomb lattice and the cross-ratios are given by
\begin{equation}\label{eqpara}
a_n=a(abc)^{n-1},\qquad b_l=b(abc)^{l-1},\qquad c_m=c(abc)^{m-1}
\end{equation}
with the notation of Figure~\ref{FigPara}~(right).
\end{theorem}

Using this parametrization, we can (locally) construct examples of conformally 
symmetric discrete immersions.

\begin{lemma}\label{lempart}
Let $a,b,c\in\C\setminus\R_{\geq 0}$. Assume that there exist three conformally 
symmetric embedded flowers with the following values of the cross-ratio 
functions $q$: $a$, $b$, $1/ab$, and $a$, $c$, $1/ac$, and $c$, $b$, $1/bc$ in 
counterclockwise orientation respectively. Define a function $\tilde{q}$ on the 
whole triangular lattice as in Figure~\ref{FigPara}~(right) by~\eqref{eqpara}.

If $|abc|=1$ and $\arg(a)+\arg(b)+\arg(c)=2\pi$, then there exists a 
generalized Doyle spiral whose cross-ratio function $q$ agrees with 
$\tilde{q}$.

If $|abc|\not=1$ or $\arg(abc)\not=2\pi$, then there exists a 
conformally symmetric discrete immersion whose cross-ratio function $q$ agrees 
with $\tilde{q}$ on this part. But this conformally symmetric discrete 
immersion cannot be extended to the whole triangular lattice.
\end{lemma}
Figure~\ref{FigAiry}~(left) shows an example of a conformally symmetric discrete 
immersion which is not a Doyle spiral.
\begin{figure}
\begin{center}
\includegraphics[clip, trim=1.2cm 0cm 0cm 0cm, height=5.3cm]{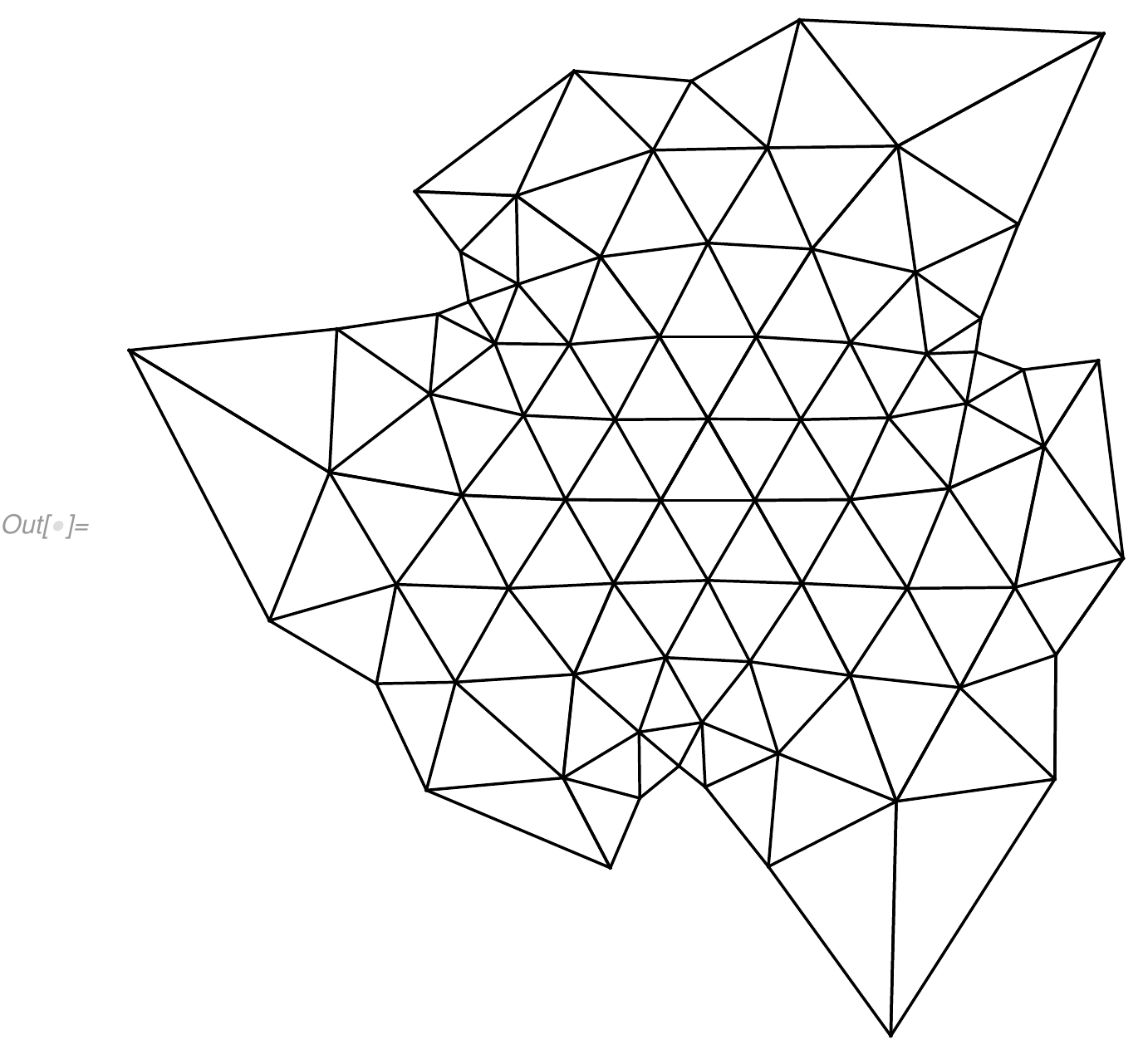}
\hspace{3em}
\includegraphics[clip, trim=1.5cm 0.5cm 1.5cm 2cm, height=5.3cm]{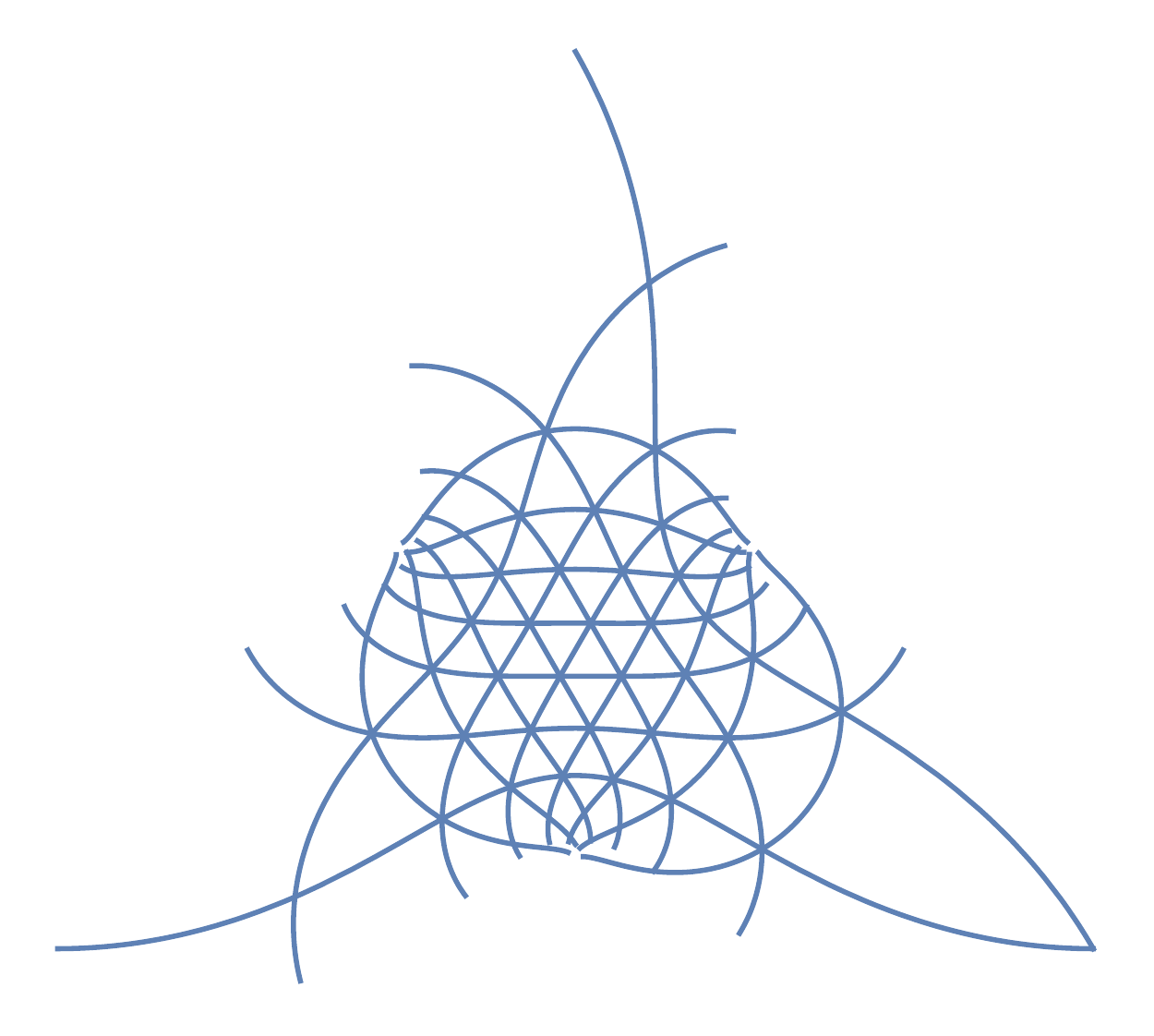}
\end{center}
 \caption{{\it Left:} Example of a general conformally symmetric triangulation 
with $a=b=c=\text{e}^{i\phi}$ and $\phi=(2/3+1/200)\pi$. The 
triangulation already starts to develop singularities. {\it 
Right:} Plot of the function $\text{e}^{i\pi/6}\frac{\text{Bi}(w) - 
\sqrt{3}\cdot \text{Ai}(w)}{\text{Bi}(w) +\sqrt{3}\cdot 
\text{Ai}(w)}$}\label{FigAiry}
\end{figure}
\begin{proof}
If $|abc|=1$ and $\arg(a)+\arg(b)+\arg(c)=2\pi$, we can start by assumption 
with an embedded conformally symmetric flower with cross-ratios $a$, $b$, $c$. 
As in the proof of Lemma~\ref{lemDoyle} we can continue to add triangles 
according to the given cross-ratios (which are constant on parallel edges) and 
obtain a generalized Doyle spiral defined on the whole lattice.

If $|abc|\not=1$ or $\arg(abc)\not=2\pi$ (or both), we still can start by 
assumption with an embedded conformally symmetric flower and we can continue 
this discrete immersion (first defined only on one flower) by using the values 
of $\tilde{q}$ for the cross-ratios. But we will show that this construction 
will always lead to non-closing flowers or to non-embedded flowers as in 
Figure~\ref{FigNonEmbedded}. 
Therefore, we cannot extend the conformally symmetric discrete immersion in 
this case.

First note that if $\arg(abc)\not=2\pi K$ for some $K\in\Z$, then we can choose 
$n,m\in\Z$ such that with the notation of Figure~\ref{FigPara} 
and~\eqref{eqpara} we have $\arg a_n,\arg c_m\in(\pi,2\pi)$. But an embedded 
flower with these cross-ratios $a_n$ and $c_m$ (and $b_{2-n-m}$) cannot exist.

If $abc=|abc|\not= 1$, we may assume that $\arg(a)+\arg(b)+\arg(c)=2\pi$ holds 
for $\arg(a)$, $\arg(b), \arg(c)\in(0,2\pi)$ because else we do not obtain an 
embedded flower (which we have by assumption).
As $|abc|\not= 1$ there exist $n,m,l$ such that $|a_n|>C$, $|b_l|>C$ and 
$|c_m|>C$ for any constant $C>0$. Therefore, if one of the arguments of $a,b,c$ 
is $>\pi$, we can choose two of the indices and consider a flower with two 
very big length cross-ratios such that the hexagon corresponding to these 
cross-ratios as in Figure~\ref{figFlowerMoeb}~(right) is not embedded. In this 
case, the corresponding flower will also be non-embedded, see 
Figure~\ref{FigNonEmbedded}. So the only remaining case is
$\arg(a),\arg(b),\arg(c)\in(0,\pi]$.
Assume without loss of generality that $\arg b\leq \arg c<\pi$. Let $n$ be such 
that $|a_n|>1/\sin^2((\pi-\arg b)/2)$ and assume that $|a_n|$ is very large. 
Consider a pair of embedded triangles with cross-ratio $a_n$ on 
the edge, then one of the adjacent angles is smaller than 
$\arcsin(1/\sqrt{|a_n|})$. Call this angle $\gamma_1$ and consider the notation 
of Figure~\ref{FigNonEmbedded}~(left). As the configuration is assumed to be 
embedded we deduce $\hat{\alpha}_1\leq \arcsin(1/|c_m|)$. Note that we can also 
choose $|c_m|$ as large as we like if we suppose that we have a conformally 
symmetric discrete immersion on the whole triangular lattice. But as 
$\pi-\alpha_1-\beta_1<\arg b$, we can choose $n$, $m$ such that $\gamma_1$, 
$\hat{\alpha}_1$ and similarly $\delta_1, \delta_2$ are so small that 
$2\arg(b)+ \gamma_1+\hat{\alpha}_1 +\delta_1 +\delta_2<2\pi$. Therefore, the 
flower cannot close up about $z_0$.
\end{proof}

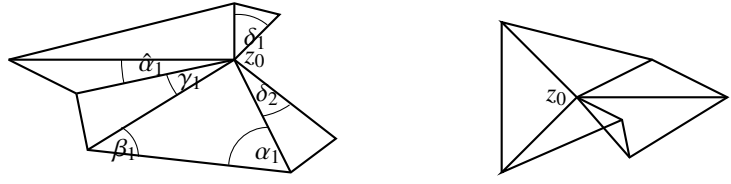
\begin{figure}
\begin{center}
 \begin{tikzpicture}[scale=1.5]
  \coordinate (A) at (0, 0);
  \coordinate [label={right:$z_0$}]  (B) at (2cm,0);
  \coordinate  (C) at (2cm, 0.5cm);
\coordinate  [label={[label distance=0.2cm]0:$\beta_1$}] (D) at (0.7cm, 
-0.8cm);
\draw  (1.cm,-0.6cm) arc (60:0:0.3cm);
\coordinate  (E) at (0.6cm, -0.3cm);
\coordinate  (F) at (2.4cm, 0.4cm);
\coordinate  [label={175:$\alpha_1$}] (G) at (2.5cm, -1cm);
\draw  (2.3cm,-0.6cm) arc (90:177:0.35cm);
\coordinate  (H) at (2.9cm, -0.7cm);

\coordinate [label={[label distance=0.3cm]190:$\gamma_1$}]  (B1) at (2cm,0);
\draw  (1.4cm,-0.13cm) arc (197:217:0.6cm);
\coordinate [label={[label distance=0.8cm]183.9:$\hat{\alpha}_1$}]  (B2) at 
(2cm,0);
\draw  (1.cm,0cm) arc (180:192:1.cm);
\coordinate [label={[label distance=0.015cm]86.9:$\delta_1$}]  (B3) at (2cm,0);
\draw  (2.cm,0.4cm) arc (95:50:0.4cm);
\coordinate [label={[label distance=0.25cm]310:$\delta_2$}]  (B4) at (2cm,0);
\draw  (2.25cm,-0.5cm) arc (280:305:0.6cm);

\draw [thick] (A) -- (C) -- (F) -- (B) -- (H) -- (G) -- (D) -- (E) ;
\draw [thick] (A) -- (B) ;
\draw [thick] (C) -- (B) ;
\draw [thick] (D) -- (B) ;
\draw [thick] (E) -- (B) ;
\draw [thick] (G) -- (B) ;
\draw [thick] (A) -- (E) ;
\end{tikzpicture}
\hspace{2cm}
 \begin{tikzpicture}
  \coordinate [label={left:$z_0$}] (A) at (0, 0);
  \coordinate  (B) at (2cm,0);
  \coordinate  (C) at (1cm, 0.5cm);
\coordinate  (D) at (0.7cm, -0.8cm);
\coordinate  (E) at (0.6cm, -0.3cm);
  \coordinate (F) at (-1cm, 1.cm);
\coordinate  (G) at (-1cm, -1.cm);

\draw [thick] (A) -- (D) -- (B) -- (C) -- (A);
  \draw [thick] (A) -- (B) ;
\draw [thick] (A) -- (E) -- (D) ;
\draw [thick] (A) -- (G) -- (E) ;
\draw [thick] (A) -- (F) -- (C) ;
\draw [thick] (F) -- (G) ;
\end{tikzpicture}
\end{center}
\caption{Examples of images of a flower about $z_0$ which does not close 
up~(left) or which contains a non embedded 
triangle~(right)}\label{FigNonEmbedded}
\end{figure}

\begin{corollary}\label{CorLat}
 Generalized Doyle spirals are the only conformally symmetric triangular meshes 
of the whole lattice.
\end{corollary}

 \subsection{Smooth analogues for conformally symmetric triangular lattices}
In the following, we are 
interested in identifying the smooth conformal maps which may be considered as 
analogues of conformally symmetric immersions.
Many examples of conformally symmetric discrete immersions are discrete 
$\vartheta$-conformal maps for some $\vartheta\in[0,\pi/2]$ which will be 
detailed in Subsection~\ref{SecEx}. 
The main idea for relating our conformally symmetric discrete immersions to 
smooth analogues is to consider the function 
$\log(q([v_0,v_k])/Q([v_0,v_k]))$
as a discrete version of the Schwarzian derivative, similarly as 
in~\cite{BH01} for conformally symmetric circle packings. 
A similar approach has been taken in~\cite{Bue17conv} where
a more detailed connection of discrete and smooth Schwarzian derivatives is 
established via convergence statements, but only for the cases of conformal 
equivalence ($\vartheta=0$) and circle patterns ($\vartheta=\pi/2$). 

Recall that $Q$ is constant on parallel edges due to the symmetry of the 
lattice. Thus generalized Doyle spirals with constant values of $q$ on parallel 
edges corresponds to a constant Schwarzian derivative and can therefore be 
considered as discrete exponentials. In the remaining case of conformally 
symmetric discrete immersions we know from~\eqref{eqpara} that $q$ is constant along lattice 
directions 
which corresponds to a linear Schwarzian derivative. In the smooth theory, 
conformal maps with linear non-constant Schwarzian derivative are quotients 
of Airy functions (see Figure~\ref{FigAiry}~(right) for an example). Airy 
functions are fundamental solutions of 
$\psi''=x\psi$, see for example~\cite{SpO87}, in particular 
\[\text{Ai}(x)=\frac{1}{\pi}\int_0^\infty\cos(xt+\frac{t^3}{3})dt,\qquad 
\text{Bi}(x)=\frac{1}{\pi}\int_0^\infty \text{e}^{xt-\frac{t^3}{3}}+ 
\sin(xt+\frac{t^3}{3})dt.\]
Therefore, conformally symmetric discrete immersions can be considered as 
discrete analogues of quotients of Airy functions.

\section{Examples of discrete $\vartheta$-conformal maps and connection to 
trigonometry}\label{SecExT}

Our study of conformally symmetric 
triangulations motivated the introduction of discrete $\vartheta$-conformal 
maps.
In this section, we
show that 
many conformally symmetric triangular lattices constitute examples of 
$\vartheta$-conformal maps. Furthermore, we show how discrete
$\vartheta$-conformal maps may be related to discrete holomorphic 
differentials. 
Finally, we present some geometric ideas related to the notion of 
$\vartheta$-conformal maps.

\subsection{Conformally symmetric triangular lattices as examples of discrete 
$\vartheta$-conformal maps}\label{SecEx}
The class of conformally symmetric triangular lattices provides examples for 
all discrete $\vartheta$-conformal maps in the sense of 
Definition~\ref{discconfdef}. 

We start with generalized Doyle spirals.
We will now prove that generalized Doyle spirals constitute examples of 
discrete $\vartheta$-conformal maps
for suitable triangular lattices and suitable choices for $\vartheta$.
 
\begin{lemma}\label{lemdiscconf}
For every generalized Doyle spiral where all quadrilaterals have been split 
consistently into triangles, there exists an intervall 
$I\subset[0,\frac{\pi}{2}]$, where $I=(\vartheta_0,\frac{\pi}{2}]$ for some 
$\vartheta_0\in[0,\frac{\pi}{2}]$ or $I=[0,\frac{\pi}{2}]$, such that for every 
$\vartheta\in I$ there is a triangular lattice $TL_\C=TL_\C(\vartheta)$ and a 
discrete $\vartheta$-conformal map $F:TL_\C\to\C$ such that the image of 
$TL_\C$ 
under $F$ is the given generalized Doyle spiral.
\end{lemma}
\begin{proof}
Given a generalized Doyle spiral, split all of its convex quadrialterals 
consistently into triangles as in Figure~\ref{FigQuad}~(right). This defines 
six angles $\alpha_1,\dots,\alpha_6\in (0,\pi)$ as in Figure~\ref{FigAng} which 
determine the quads of the Doyle spiral up to Euclidean motions and scaling.
Due to the convexity of the quad, these six angles satisfy 
$\alpha_3+\alpha_6<\pi$, $\alpha_2+\alpha_4<\pi$,
$\alpha_1+\alpha_2+\alpha_3=\pi$, and $\alpha_4+\alpha_5+\alpha_6=\pi$ and they 
determine the cross-ratios $q_1,q_2,q_3$ of the Doyle spiral via
\begin{align*}
 \log q_1 &=\log\left(\frac{\sin\alpha_6}{\sin\alpha_4}\cdot 
\frac{\sin\alpha_2}{\sin\alpha_3}\right) +i(\alpha_1+\alpha_5),\\
 \log q_2 &=\log\left(\frac{\sin\alpha_3}{\sin\alpha_1}\cdot 
\frac{\sin\alpha_4}{\sin\alpha_5}\right) +i(\alpha_2+\alpha_6),\\
 \log q_3 &=\log\left(\frac{\sin\alpha_5}{\sin\alpha_6}\cdot 
\frac{\sin\alpha_1}{\sin\alpha_2}\right) +i(\alpha_3+\alpha_4).
\end{align*}
The labeling of the cross-ratios is defined according to these formulas.
\begin{figure}
\begin{tikzpicture}[rotate=-30,scale=1]
\tkzDefPoint(1,0){A}
\tkzDefPoint(2,0.3){B}
\tkzDefPoint(0,0){O}
\tkzDefPoint({2*2.022*cos(8.531*pi/180)-0.3*2.022*sin(8.531*pi/180)},{
0.3*2.022*cos(8.531*pi/180)+2*2.022*sin(8.531*pi/180)}){C}
\tkzDrawSegment[style=thick](A,B)
\tkzDrawSegment[style=thick](B,C)
\tkzDefPointBy[rotation=center O angle 23](A)
\tkzGetPoint{A1}
\tkzDefPointBy[homothety=center O ratio 1.4](A1)
\tkzGetPoint{A1'}
\tkzDefPointBy[rotation=center O angle 23](B)
\tkzGetPoint{B1}
\tkzDefPointBy[homothety=center O ratio 1.4](B1)
\tkzGetPoint{B1'}
\tkzDefPointBy[rotation=center O angle 23](C)
\tkzGetPoint{B1}
\tkzDefPointBy[homothety=center O ratio 1.4](C1)
\tkzGetPoint{C1'}

\tkzDrawSegment[style=thick](A1',B1')
\tkzDrawSegment[style=dotted, color=red](A,A1')
\tkzDrawSegment[style=dotted, color=red](B,B1')
\tkzDrawSegment[style=dotted, color=red](C,C1')
\tkzDrawSegment[style=dashed, color=blue](A1',B)
\tkzDrawSegment[style=thick](B1',C1')
\tkzDrawSegment[style=dashed, color=blue](B1',C)

\tkzMarkAngle[arc=l, size= .6cm](C,B,B1')
\tkzLabelAngle[pos = .4](C,B,B1'){$\alpha_1$}
\tkzMarkAngle[arc=l, size= .7cm](B1',C,B)
\tkzLabelAngle[pos = -0.5](B1',C,B){$\alpha_2$}
\tkzMarkAngle[arc=l, size= .5cm](B,B1',C)
\tkzLabelAngle[pos = .4](B,B1',C){$\alpha_3$}
\tkzMarkAngle[arc=l, size= .7cm](C1',C,B1')
\tkzLabelAngle[pos = .4](C1',C,B1'){$\alpha_4$}
\tkzMarkAngle[arc=l, size= 1.4cm](B1',C1',C)
\tkzLabelAngle[pos = 1.2](B1',C1',C){$\alpha_5$}
\tkzMarkAngle[arc=l, size= .5cm](C,B1',C1')
\tkzLabelAngle[pos = .4](C,B1',C1'){$\alpha_6$}
\tkzMarkAngle[arc=l, size= .35cm](B1',B,A1')
\tkzLabelAngle[pos = .15](B1',B,A1'){$\alpha_4$}
\tkzMarkAngle[arc=l, size= .7cm](A1',B1',B)
\tkzLabelAngle[pos = .5](A1',B1',B){$\alpha_5$}
\tkzMarkAngle[arc=l, size= .4cm](B,A1',B1')
\tkzLabelAngle[pos = .3](B,A1',B1'){$\alpha_6$}
\end{tikzpicture}
\hfill
\includegraphics[width=0.4\textwidth]{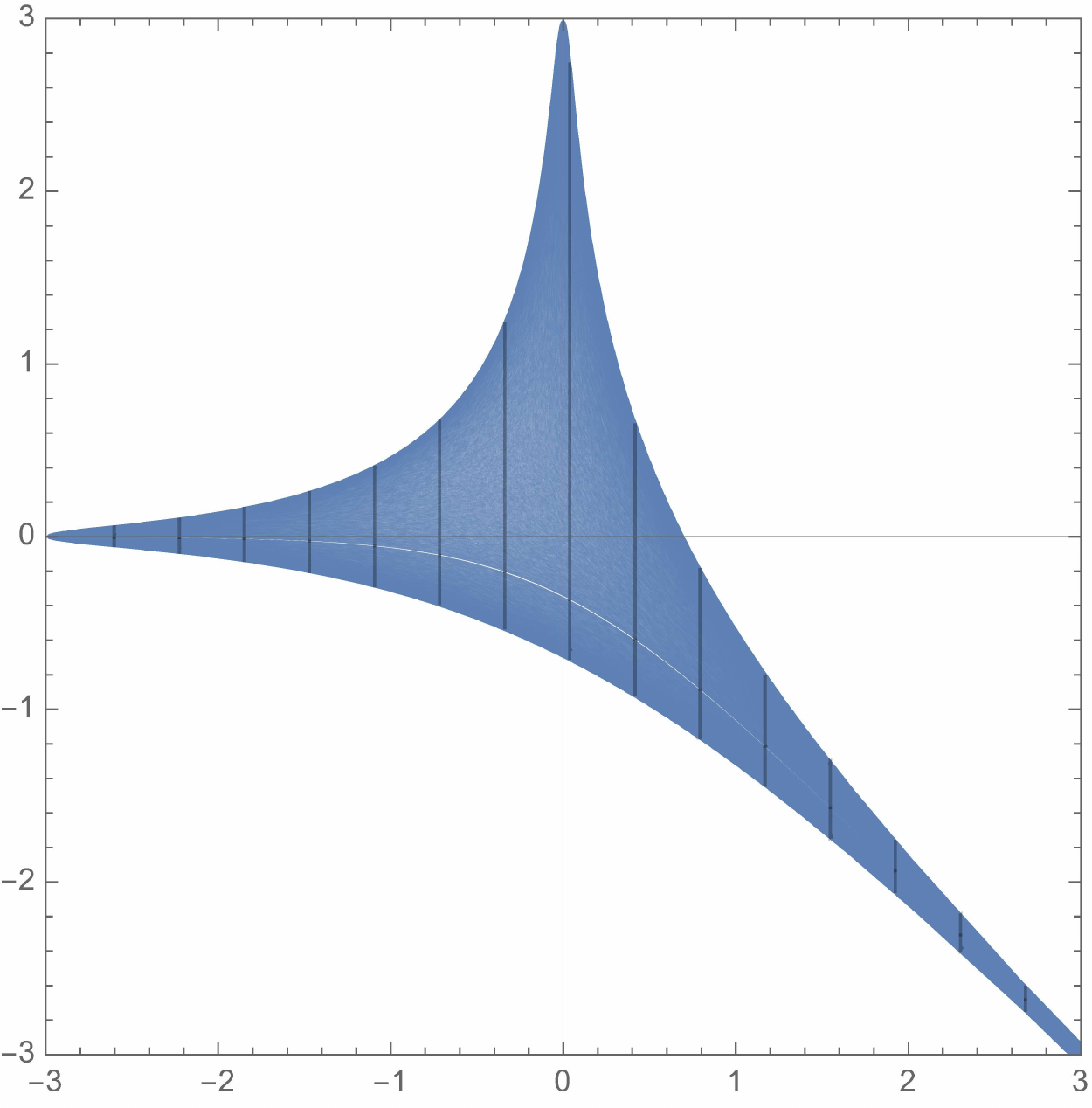}
\caption{{\it Left:} Definition of angles in the splitted convex quads of a 
generalized Doyle spiral; {\it right:} Image domain of $G(\alpha,\beta)$
}\label{FigAng}
\end{figure}

Let $TL_\C$ be a triangular lattice with angles $\alpha,\beta,\gamma\in 
(0,\pi)$ 
as in Figure~\ref{FigTiling}, where $\alpha+\beta+\gamma=\pi$. Then its three 
cross-ratios may be expressed in terms of the angles as
\begin{align*}
 \log Q_1 &=2\log\frac{\sin\beta}{\sin\alpha} +2i\gamma, &
 \log Q_2 &=2\log\frac{\sin\alpha}{\sin\gamma} +2i\beta, &
 \log Q_3 &=2\log\frac{\sin\gamma}{\sin\beta} +2i\alpha.
\end{align*}
Thus, for $\vartheta=\pi/2$ we can choose $\gamma=(\alpha_1+\alpha_5)/2$, 
$\beta=(\alpha_2+\alpha_6)/2$, and $\alpha=(\alpha_3+\alpha_4)/2$ and obtain 
$\text{Re}(\text{e}^{-i\frac{\pi}{2}}q_k)= 
\text{Re}(\text{e}^{-i\frac{\pi}{2}}Q_k)$ for $k=1,2,3$. Also, at least for 
$\vartheta$ in a small neighborhood of $\pi/2$, the equations 
$\text{Re}(\text{e}^{-i\vartheta}q_k)= 
\text{Re}(\text{e}^{-i\vartheta}Q_k)$ are still solvable in terms of 
$\alpha,\beta,\gamma$ for given $\alpha_1,\dots,\alpha_6$ as above. But the 
image of 
$\{(\alpha,\beta)\in (0,\pi)^2 : \alpha+\beta<\pi\}$ under 
\[G(\alpha,\beta)=\begin{pmatrix} \log\frac{\sin\beta}{\sin\alpha}
\\[1ex]
\log\frac{\sin\alpha}{\sin(\alpha+\beta)}\end{pmatrix}\]
is not convex, see Figure~\ref{FigAng}~(right) (compare also 
to~\cite[Fig.~9]{BPS13}). Thus, there exist configurations with angles 
$\alpha_1,\dots,\alpha_6\in (0,\pi)$ as in Figure~\ref{FigAng} such that the 
system $G(\alpha_4,\alpha_6)+G(\alpha_3,\alpha_2) =2G(\alpha,\beta)$ has no 
solution $\alpha,\beta\in (0,\pi)$ with $\alpha+\beta<\pi$. This is equivalent 
to the fact that for such a configuration the system of equations 
$\text{Re}(\log q_k)=\text{Re}(\log Q_k)$, $k=1,2,3$ have no solution for 
$Q_k$, which 
means that there is no corresponding triangular lattice $TL_\C$ such that the 
discrete immersion from $TL_\C$ to the generalized Doyle spiral is discrete 
$\vartheta$-conformal for this $\vartheta$. Thus, depending on the given 
generalized Doyle spiral, 
the interval for suitable $\vartheta$'s may in some cases be strictly smaller 
than $[0,\pi/2]$.
\end{proof}

\begin{remark}
 In the theory of conformally equivalent triangular lattices, a change of the 
combinatorics by means of edge flipping is required in order to guarantee the 
existence of solutions, see~\cite{GGLSW18,Sp17}. Lemma~\ref{lemdiscconf} 
suggests that such combinatorial changes may also be necessary for discrete
$\vartheta$-conformal maps. 
This demonstrates that discrete $\vartheta$-conformal maps, which interpolate between 
circle patterns and conformally equivalent triangulations, inherit
properties from both theories. 
A change of the 
combinatorics could also prevent triangles and flowers from 
being degenerate as considered in the proof of Lemma~\ref{lempart}.
\end{remark}

 The proof of Lemma~\ref{lemdiscconf} actually shows that if the angles of the 
split quads are all in a suitable neighborhood of $\pi/3$ then the 
corresponding 
generalized Doyle spirals are discrete $\vartheta$-conformal maps for all 
$\vartheta\in[0,\pi/2]$. This observation motivated the introduction and study 
of discrete $\vartheta$-conformal maps.

Furthermore, we can map every flower of a generalized Doyle spiral by a 
M\"obius transformation to a symmetric embedded flower in the sense that 
${\mathcal M}(z_j)=-{\mathcal M}(z_{j+3})$. Then we can observe that the generalized 
Doyle spiral is 
discrete $\vartheta$-conformal for every $\vartheta\in[0,\pi/2]$ if and only if 
this symmetric flower is convex.

Apart from Doyle spirals, we know from Lemma~\ref{lempart} that a conformally 
symmetric discrete immersion is only defined on a part $TL$ of the lattice 
$TL_\C$ and its cross-ratio function $q$ is given by~\eqref{eqpara} on the 
corresponding part of the hexagonal lattice indicated in
Figure~\ref{FigPara}. In particular, the proof of Theorem~\ref{theoabc} shows 
that for the part of the lattice where the conformally 
symmetric discrete immersion is defined, we have $\log a_n=\log 
a+(n-1)\log(abc)$ and analogous formulas for $\log b_l$ and $\log c_m$.

Therefore we deduce that
\begin{itemize}
 \item if $|abc|=1$, then $\text{Re}(\log a_n)= \text{Re}(\log a)=|a|$  and 
analogously for $\log b_l$ and $\log c_m$. If there exists a triangular lattice 
with length cross-ratios $|a|,|b|,|c|$ (see the proof of 
Lemma~\ref{lemdiscconf} for this issue), then the conformally 
symmetric discrete immersion is discrete $\vartheta$-conformal for 
$\vartheta=0$.
\item if $abc\in\R_+=(0,\infty)$ or equivalently 
$\arg(a)+\arg(b)+\arg(c)=2\pi$, the circle 
pattern obtained from adding the circumcircles to all triangles has the same 
intersection angles as the circle pattern built from the circumcircles of the 
triangles in the lattice $TL_\C$ with angles $\alpha=\arg(a)/2$, 
$\beta=\arg(b)/2$ and $\gamma=\arg(c)/2$. In this case, the conformally 
symmetric discrete immersion is discrete $\vartheta$-conformal for 
$\vartheta=\pi/2$.
\item if $\text{Re}(\text{e}^{-i\vartheta} \log(abc))=0$ or equivalently 
$\cos\vartheta\log|abc|+\sin\vartheta\arg(abc)=0$ for some 
$\vartheta\in(0,\pi/2)$, this may be an example of a general discrete 
$\vartheta$-conformal map. In fact, this is the case if 
there exists a triangular lattice $TL_\C$ whose cross-ratio function $Q$ 
satisfies
$\text{Re}(\text{e}^{-i\vartheta} \log Q_1)= \text{Re}(\text{e}^{-i\vartheta} 
\log a)$, $\text{Re}(\text{e}^{-i\vartheta} \log Q_2)= 
\text{Re}(\text{e}^{-i\vartheta} \log b)$, and 
$\text{Re}(\text{e}^{-i\vartheta} 
\log Q_3)= \text{Re}(\text{e}^{-i\vartheta} \log c)$, where $Q_1,Q_2,Q_3$ 
denote the three (in general different) values of the cross-ratio function $Q$ 
on the lattice $TL_\C$.
\end{itemize}

Combining this reasoning with Corollary~\ref{CorLat}, we obtain the following 
immediate consequence.
\begin{corollary}
 Generalized Doyle spirals are the only conformally symmetric triangular meshes 
from discrete $\vartheta$-conformal maps 
of the whole lattice.
\end{corollary}

\subsection{Connection to discrete holomorphic quadratic 
differentials}\label{remquaddiff}

For any smooth one parameter family $g(t):T\to\C$, $t\in(-\varepsilon, \varepsilon)$,
of discrete immersions of a
simply connected triangulation the considerations and results for 
infinitesimal deformations of~\cite[Section~5]{Lam2016} apply. In 
particular, the corresponding cross-ratio functions 
$q(t)$ on the interior edges is a perturbation of 
$q(0)\equiv Q$ 
for some general immersed triangulation $T$. The logarithmic 
derivatives
\begin{equation}\label{eqdefqij}
q_{ij}:=\frac{1}{q(0)([v_i,v_j])} 
\left.\frac{d}{dt}\right|_{t=0} q(t)([v_i,v_j])
\end{equation}
on the interior edges $[v_i,v_j]\in E$ satisfy
\begin{align}\label{eqquaddiff}
 \sum\limits_{v_j:v_j\text{ adjacent to } v_i} q_{ij}&=0 & \text{and}\qquad
 \sum\limits_{v_j:v_j\text{ adjacent to } v_i} \frac{q_{ij}}{v_i-v_j}&=0.
\end{align}
This is a consequence of the generalized versions of~\eqref{eqq1} 
and~\eqref{eqq2} explained in Remark~\ref{remGen} by taking derivatives, see 
also~\cite[Cor.~5.3]{Lam2016}.
Furthermore, for a family $g(t)$ of discrete $\vartheta$-conformal maps we 
deduce that $q_{ij}\in i\text{e}^{i\vartheta}\R$. This is analogous to the two 
cases for conformal equivalence ($q_{ij}\in i\R$) and circle patterns 
($q_{ij}\in \R$) considered in~\cite{Lam2016}. Therefore we may also call 
$q_{ij}$ a {\em discrete holomorphic quadratic differential}. In view 
of~\cite{Lam18}, we can interpret the cases for $\vartheta\not=0,\pi/2$ as 
corresponding to the associated family of minimal surfaces.

For the special example of generalized Doyle spirals,
a one parameter family of discrete  
$\vartheta$-conformal maps which are generalized Doyle spirals may be easily 
constructed explicitly. 
Start with a 
triangular lattice $TL$ with cross-ratio function 
$Q:EL\to\{A,B,C\}\subset\C\setminus\R_{\leq 0}$. In particular, $ABC=1$. For 
$\vartheta\in[0,\pi/2]$, set $R(A,\vartheta)=\text{Re} 
[\text{e}^{-i\vartheta}\log(A)]$ and $I(A,\vartheta)=\text{Im}
[\text{e}^{-i\vartheta}\log(A)]$ and define $R(B,\vartheta)$, 
$R(C,\vartheta)$ and $I(B,\vartheta)$, $I(C,\vartheta)$ analogously. Let 
$a,b,c\in\C\setminus\R_{\leq 0}$ be 
the cross-ratio function of another triangular lattice. Then there exists a one 
parameter family of generalized Doyle spirals $\widehat{T}(t)$ with cross-ratio 
function $q(t):EL\to\{q_A(t),q_B(t),q_C(t)\}$, where 
\[q_A(t)=\text{exp}(R(A,\vartheta)\text{e}^{i\vartheta} 
+i\text{e}^{i\vartheta}( 
t\text{Im}[\text{e}^{-i\vartheta}\log(a)]+(1-t)I(A,\vartheta)))\] 
and $q_B(t)$ and $q_C(t)$ are defined analogously.  
The logarithmic derivatives $q_{ij}$ defined by~\eqref{eqdefqij} 
satisfy~\eqref{eqquaddiff} and additionally
$\sum\limits_{[v_j,v_j]\text{ incident to triangle } \Delta} q_{ij}=0$ for all 
triangular faces $\Delta$
of $TL_\C$. This last property means that this discrete holomorphic quadratic 
differential is {\em integrable}. 

\begin{corollary}
Every integrable discrete holomorphic quadratic differential on a triangular 
lattice $TL_\C$ arises from a one parameter family of discrete  
$\vartheta$-conformal maps which are generalized Doyle spirals.
\end{corollary}

\subsection{Trigonometry with parameter $\vartheta$}\label{secGeomTriangle}

It is well known that a triangle is determined up to similarity transformations by its angles or by 
the ratios of its edge lengths. 
These two well known examples will correspond to our cases $\vartheta=0$ and 
$\vartheta=\pi/2$. For 
$\vartheta\in(0,\pi/2)$ we switch to unusual parameters for the 
construction of triangles.

As for the definition of the cross-ratio in~\eqref{eqdefQ}, we mainly focus on notions based on 
differences like 
$w_{ij}=v_j-v_i$ etc. These are 
associated to directed edges of the embedded non-degenerate counterclockwise 
oriented trianlges $\Delta[v_i,v_j,v_k]$ of the given triangulation $T$. 
Similarly, we consider the directed edges $z_{ij}=z_j-z_i$ of the 
non-degenerate image 
triangles $F(\Delta[v_i,v_j,v_k])= \Delta[z_i,z_j,z_k]$ for any 
discrete immersion $F$. These complex numbers satisfy the obvious closing 
condition $z_{ij}+z_{jk}+z_{ki}=0$. This algebraic equation is 
naturally studied in $\CP^2$ and this means that we consider triangles up to global
similarity transformations which does not change the cross-ratios $q$. When 
dealing with similarity classes of embedded triangles, one natural approach is 
to work with the quotients
$\tau_k^{ij}=z_{kj}/z_{ki}\in\Ha=\{z\in\C :\text{Im}(z)>0\}$ in the upper 
half-plane of $\C$ or their
logarithms $\log \tau_k^{ij}$ as for example in~\cite[Section~1.5]{CM12}. They 
satisfy the relations
 $\tau_k^{ij}\tau_i^{jk}\tau_j^{ki}=-1$ and
$\tau_j^{ki}\tau_i^{jk}= \tau_k^{ij}-1$.
Thus, if we denote $\tau_k^{ij}=\tau$ we have $\tau_i^{jk}=1/(1-\tau)$ and 
$\tau_j^{ki}=1-1/\tau$. Furthermore, their logarithms satisfy
\begin{align}\label{eqlogtau}
 \log\tau_k^{ij}+ \log\tau_i^{jk}+ \log\tau_j^{ki}=\pi i.
\end{align}
In view of~\eqref{eqdefdiscconf}, we choose to work with these logarithmic 
variables. As scale-rotations may be parametrized via $z\mapsto z\cdot 
\text{e}^A$ for $A\in\C$, these variables belong to 
the Lie algebra of scale-rotations and
are usually split into real and imaginary parts which correspond to scaling
and rotation respectively. The right hand side of~\eqref{eqlogtau} fits to this 
splitting. For our purposes we choose another orthonormal basis 
of this Lie algebra, namely $\text{e}^{i\vartheta}$ and 
$i\text{e}^{i\vartheta}$. To simplify further considerations, we introduce 
parameters relative to a given configuration. In particular, instead of 
$\log\tau_k^{ij}= \log(z_{kj}/z_{ki})$, we focus on 
$\zeta_k^{ij}= \log\left( \frac{z_{kj}}{z_{ki}} \frac{w_{ki}}{w_{kj}} \right) := \log 
\frac{z_{kj}}{z_{ki}} -\log \frac{w_{kj}}{w_{ki}}$. 
These variables satisfy 
\begin{align}\label{eqzeta}
\zeta_k^{ij}+ \zeta_i^{jk}+ \zeta_j^{ki}= 0.
\end{align}
 and we will 
split them according to the orthonormal basis $\text{e}^{i\vartheta}$ and 
$i\text{e}^{i\vartheta}$, that is $\zeta_k^{ij}= \log\left(  
\frac{z_{kj}}{z_{ki}} \frac{w_{ki}}{w_{kj}} \right)
=\text{e}^{i\vartheta} (i\hat{\zeta}_{ij}+\hat{\hat{\zeta}}_{ij})$ with real 
$\hat{\zeta}_{ij},\hat{\hat{\zeta}}_{ij}\in\R$. 
In the triangular graph $G^\Delta_M$ whose vertices are the midpoints 
$M_{ij}=(v_i+v_j)/2$ of the edges $[v_i,v_j]$ of the triangle 
$\Delta[v_i,v_j,v_k]$, the variables $\hat{\zeta}_{ij}= -\hat{\zeta}_{ji}$ and 
$\hat{\hat{\zeta}}_{ij}= -\hat{\hat{\zeta}}_{ji}$ may be interpreted as a 
$1$-form on the edges $[M_{kj},M_{ki}]$. Condition~\eqref{eqzeta} shows that 
this $1$-form is closed. Therefore, we can integrate and obtain 
$\zeta_k^{ij}= \text{e}^{i\vartheta}(i(\omega_{kj}-\omega_{ki}) 
+(\nu_{kj}-\nu_{ki}))$. The real variables $\omega_{ij},\nu_{ij}\in\R$ 
are associated to the midpoints $M$ and thus to the original edges of the 
triangle $\Delta[v_i,v_j,v_k]$. They 
are unique up to a common constant, respectively. Furthermore, if we switch 
back to our original triangles $\Delta[v_i,v_j,v_k]$ and $\Delta[z_i,z_j,z_k]$, 
we deduce that by a suitable choice of these constants we can express
\begin{equation}\label{defoij}
 z_{ij}=w_{ij}\cdot \text{exp}(\text{e}^{i\vartheta}(i\omega_{ij}+\nu_{ij})).
\end{equation}
Thus, $\text{exp}(\text{e}^{i\vartheta}(i\omega_{ij}+\nu_{ij}))$ describes the 
scale-rotation which transforms $w_{ij}$ into $z_{ij}$.
For $\vartheta=\pi/2$, the variables $\omega_{ij}=\log|z_{ij}|-\log|w_{ij}|$ 
and $\nu_{ij} =\arg(z_{ij})-\arg(w_{ij})$ are just the 
(negative) logarithmic scale factor of the edge $z_{ij}$ and its relative 
rotation angle with respect to $w_{ij}$. This holds also true for $\vartheta=0$ 
where the roles of logarithmic scale factor and rotation angle are interchanged. 
In fact for $\vartheta\in(0,\pi/2)$, the new variables $\omega_{ij},\nu_{ij}$ 
are linear combinations of these two known parameters.
Note that the values of $\omega_{ij}$ and/or $\nu_{ij}$ are not uniquely 
defined in $\R$ by~\eqref{defoij}. The same result is obtained for 
$\omega_{ij}+2\pi\cos\vartheta\cdot K$ and $\nu_{ij}+2\pi\sin\vartheta\cdot K$ 
for any $K\in\Z$. In order to simplify our calculations, we will always 
choose the values $\omega_{ij},\nu_{ij}$ such that the angle 
$\alpha_k^{\Delta[z_i,z_j,z_k]}$ in the image 
triangle $\Delta[z_i,z_j,z_k]$ at the vertex $z_k$ can be expressed as
\begin{equation}\label{eqdefalpha}
 \alpha_k^{\Delta[z_i,z_j,z_k]}= \alpha_k^{\Delta[v_i,v_j,v_k]}+ 
\cos\vartheta(\omega_{ki}-\omega_{jk}) +\sin\vartheta(\nu_{ki}-\nu_{jk}).
\end{equation}
This defines our variables a the given triangle uniquely up to a global 
constant, which will not be 
important for our considerations as we mainly work with differences.

\section{A variational principle for discrete $\vartheta$-conformal 
maps}\label{SecVari}

Recall that $T$ is an immersed triangulation in the plane. For simplicity, we 
assume that $T$ is simply connected. Let $F:T\to\C$ be
a discrete immersion with image $\widehat{T}$.
In this section we prove the existence of a variational principle 
for discrete $\vartheta$-conformal maps. 

Similarly as for circle patterns and conformally equivalent triangulations, 
discrete $\vartheta$-conformal maps are associated to a special 
description of triangles which has been introduced in 
Section~\ref{secGeomTriangle}. We use this for a reformulation of 
condition~\eqref{eqdefdiscconf} which in turn will be shown to be 
a partial derivative of a locally defined functional 
${\mathcal F}_\vartheta(\omega)$.
Thus we aim at the following theorem, which is a direct consequence of
Theorems~\ref{theoclosed}--\ref{theoconstruct} below.

\begin{theorem}\label{theovari}
 For any discrete $\vartheta$-conformal map $F:T\to\widehat{T}$, the 
corresponding variables $\omega\in\R^{|E|}$ constitute the unique maximizer (up 
to a global constant) of a locally concave functional ${\mathcal F}_\vartheta$.

Conversely, every maximizer of ${\mathcal F}_\vartheta$ on a simply connected 
triangulation $T$ corresponds to a $\vartheta$-conformal map 
$F:T\to\widehat{T}$.
\end{theorem}

\subsection{A single triangle in new variables}\label{subsecvari}

As in Section~\ref{secGeomTriangle}, let $w_{ij}=v_j-v_i$ denote the directed 
edges of an embedded non-degenerate trianlges 
$\Delta[v_i,v_j,v_k]$. Similarly, we consider 
the directed edges $z_{ij}=z_j-z_i$  of the non-degenerate image 
triangles $F(\Delta[v_i,v_j,v_k])= \Delta[z_i,z_j,z_k]$ for any 
discrete immersion $F$. Now, we express 
the obvious closing condition $z_{ij}+z_{jk}+z_{ki}=0$  as
\begin{equation}\label{eqtriangleom}
 w_{ij}\cdot \text{exp}(\text{e}^{i\vartheta}(i\omega_{ij}+\nu_{ij})) 
+w_{jk}\cdot \text{exp}(\text{e}^{i\vartheta}(i\omega_{jk}+\nu_{jk}))  
+w_{ki}\cdot \text{exp}(\text{e}^{i\vartheta}(i\omega_{ki}+\nu_{ki})) =0.
\end{equation}
In the following, we consider the $\omega$'s as our main variables and the 
$\nu$'s as dependent functions.
By the implicit function theorem, we can assume that the $\nu$'s 
depend smoothly on the given variables $\omega$ in a suitable neighborhood 
of a solution. 

As in the cases for given edge 
lengths or angles, not all choices of $\omega$'s correspond to a 
triangle. This means, that for some choices of $\omega$'s there do 
not exist any real $\nu$'s  such that~\eqref{eqtriangleom} holds.
We briefly describe the set of suitable choices of $\omega$'s for the 
normalized case $z_{ij}=w_{ij}=1$, $z_{jk}=w_{jk}\cdot 
\text{exp}(\text{e}^{i\vartheta}(i\omega_{1}+ 
\nu_{1})) =:w_{jk}\cdot \sigma_1$ and $z_{ki}=w_{ki}\cdot
\text{exp}(\text{e}^{i\vartheta}(i\omega_{2}+ \nu_{2}))=:w_{ki}\cdot \sigma_2$.

Associate to a non-zero complex number $w$ 
the set $S_w=\{(\log|w|,\arg w+2\pi k) : k\in\Z\}\subset\R^2$ of all 
values of the logarithm $\log w$ (i.e.\ possible polar 
coordinates). We apply the map $w\mapsto S_w$ to the edge of the image
triangle and consider in particular
$L:(\sigma_{1},\sigma_{2})\mapsto S_{\sigma_{1}}\times S_{\sigma_{2}}$. 
Then the set  $E_\vartheta$
of all suitable choices for $\omega_1,\omega_2$ is the intersection of the 
plane in $\C^2\cong\R^4$
generated by $\left(\begin{smallmatrix} -\sin\vartheta \\ \cos\vartheta 
\end{smallmatrix}\right)\times \{0\}$ and $\{0\}\times 
\left(\begin{smallmatrix} -\sin\vartheta \\ \cos\vartheta 
\end{smallmatrix}\right)$ with the image of the set 
$\{(\sigma_{1},\sigma_{2})\in\C^2: 1+w_{jk}\cdot 
\sigma_{1}+w_{ki}\cdot \sigma_{2}=0 \text{ and } w_{jk}\cdot 
\sigma_{1},-w_{ki}\cdot \sigma_{2}\in\Ha \}$ 
under the map $L$.
As all sets $S_w$ have a translational
period of $2\pi$ in the second component, the set $E_\vartheta$
is $(2\pi\cos\vartheta)$-periodic if $\vartheta\not= \pi/2$.
By the inverse function theorem, we easily deduce that the set 
$E_\vartheta$ is open.

\begin{remark}
 Consider as above the normalized directed edges $z_{ij}=w_{ij}=1$, 
$z_{jk}=w_{jk}\cdot \text{exp}(\text{e}^{i\vartheta}(i\omega_{1}+ 
\nu_{1}))$, $z_{ki}=w_{ki}\cdot
\text{exp}(\text{e}^{i\vartheta}(i\omega_{2}+ \nu_{2}))$ which (possibly) build 
a non-degenerate triangle with counterclockwise orientation.  
In this description, the existence of suitable $\nu_1,\nu_2$ for given 
$\omega_1,\omega_2$ is equivalent to the fact that the two spirals $s_1(t)= 
a\cdot \text{exp}(\text{e}^{i\vartheta}t)$ and $s_2(t)= b\cdot 
\text{exp}(\text{e}^{i\vartheta}t)$, where $a=w_{jk}\cdot 
\text{exp}(i\text{e}^{i\vartheta}\omega_{1})$, $b=w_{ki}\cdot
\text{exp}(i\text{e}^{i\vartheta}\omega_{2})$ and $t\in\R$, intersect 
transversally in the upper half-plane $\Ha$. Depending on $a$ and $b$, 
there may exist several intersection points for some values of 
$\omega_1,\omega_2$. In these cases, we obtain several (smooth) functions 
$\nu_j(\omega)$.

If we only consider (small) variations of a given configuration, we may assume 
that $\omega$, and thus $\nu$, is in a neighborhood of zero and therefore 
obtain a well-defined function $\nu(\omega)$.
In the general case, we have to consider several functions $\nu^{(l)}(\omega)$ 
on different, overlapping domains.
\end{remark}

\subsection{Condition for discrete $\vartheta$-conformal maps in new 
variables}

Our next goal is to reformulate condition~\eqref{eqdefdiscconf}. To this end, 
it is sufficient to consider the embedded configuration of two adjacent 
triangles $\Delta[v_i,v_j,v_k]$ and  
$\Delta[v_i,v_l,v_j]$ and their images $F(\Delta[v_i,v_j,v_k])= 
\Delta[z_i,z_j,z_k]$ and $F(\Delta[v_i,v_l,v_j])= 
\Delta[z_i,z_l,z_j]$ for any discrete immersion $F$. The cross-ratios
$q_{ij}=q([v_i,v_j])$ defined by~\eqref{eqdefq} can be expressed as
\begin{equation}
 q_{ij}= \frac{w_{il}w_{jk}}{w_{lj}w_{ki}} \cdot 
\frac{\text{exp}(\text{e}^{i\vartheta}(i\omega_{il}+ \nu_{il}))}{
\text{exp}(\text{e}^{i\vartheta}(i\omega_{lj}+ \nu_{lj}))} \cdot 
\frac{\text{exp}(\text{e}^{i\vartheta}(i\omega_{jk}+ \nu_{jk}))}{
\text{exp}(\text{e}^{i\vartheta}(i\omega_{ki}+ \nu_{ki}))}.
\end{equation}
As $Q([v_i,v_j])= \frac{w_{il}w_{jk}}{w_{lj}w_{ki}}$ we deduce that 
\begin{equation*}
 \text{e}^{-i\vartheta}\log q_{ij}= \text{e}^{-i\vartheta}\log Q([v_i,v_j])
+i (\omega_{il}-\omega_{lj}+\omega_{jk}-\omega_{ki}) + \nu_{il}-\nu_{lj} 
+\nu_{jk}-\nu_{ki}.
\end{equation*}
Thus $F$ is discrete $\vartheta$-conformal on this minimal example of two 
adjacent triangles if and only if
\begin{equation}\label{eqcondnu}
 \nu_{il}-\nu_{lj} +\nu_{jk}-\nu_{ki}=0
\end{equation}
holds. 

Recall that the $\nu$'s depend on the variables $\omega$ in the respective 
triangles. We will emphasize this by writing $\nu_{il}^{\Delta[v_i,v_l,v_j]}$ 
etc. In particular, the two values of $\nu$ associated to the same 
edge will in general be different: $\nu_{ij}^{\Delta[v_i,v_l,v_j]} \not= 
\nu_{ij}^{\Delta[v_i,v_j,v_k]}$.

\subsection{Brief review on the known cases for $\vartheta=0$ and 
$\vartheta=\pi/2$}\label{SecReview}

The abstract variables $\omega_{ij}$ and $\nu_{ij}$ and 
condition~\eqref{eqcondnu} become geometrically more specific if 
we consider the case of circle patterns ($\vartheta=\pi/2$) and conformally 
equivalent triangulations ($\vartheta=0$). As in the general case studied in 
Section~\ref{secGeomTriangle}, we regard the embedded, counterclockwise oriented 
triangle 
$\Delta[z_i,z_j,z_k]$ as obtained from a given (embedded, counterclockwise 
oriented) triangle $\Delta[v_i,v_j,v_k]$ and consider the relativ changes of 
the directed edges, edge lengths and angles.

\subsubsection{Circle patterns}
In the case $\vartheta=\pi/2$, the free variables $\omega_{ij}$ are the 
logarithmic length changes $-\omega_{ij}=  \log|z_{ij}|-\log|w_{ij}|$ 
of the edges. The dependent variables are the relative counterclockwise  
rotation angles 
$\nu_{ij} =\arg(z_{ij})- \arg(w_{ij})$ (modulo $2\pi$) of the edges. For two 
incident edges in a triangle $\Delta[z_i,z_j,z_k]$, the differences 
$\nu_{ki}-\nu_{jk}$  (modulo $2\pi$) give the change of the angle 
$\alpha_k^{\Delta[z_i,z_j,z_k]} -\alpha_k^{\Delta[v_i,v_j,v_k]}$ opposite to 
the edge $ij$. Thus condition~\eqref{eqcondnu} expresses in this case that the 
sum of opposite angles in the two incident triangles sharing an edge does 
not change. This encodes the usual condition for circle patterns that 
intersection angles are preserved.

Let $\Phi_{ij}$ be the sums of opposite angles in two adjacent embedded
triangles sharing the edge $[v_i,v_j]$. Recall that this is the exterior 
intersection angle between the corresponding circumcircles of the triangles. 
For 
a boundary edge we set $\Phi_{ij}$ equal to the opposite angle.
Then as detailed in~\cite[App.~C]{BPS13}, a corresponding concave functional is 
given by
\[{\mathcal F}_{\frac{\pi}{2}}(\omega)= -\sum\limits_{\Delta[v_i,v_j,v_k]} 
\hat{V}(\log|w_{ij}|-\omega_{ij}, \log|w_{jk}|-\omega_{jk}, 
\log|w_{ki}|-\omega_{ki}) - \sum\limits_{[v_i,v_j]} \Phi_{ij}\omega_{ij},\]
where the first sum is taken over all triangles and the 
second sum over all edges of the given triangulation $T$. 

The function $\hat{V}$ is defined for variable $(\rho_{12}, 
\rho_{23},\rho_{31})\in \R^3$ such that the positive numbers 
$\ell_{12}=\text{e}^{\rho_{12}}$, $\ell_{23}=\text{e}^{\rho_{23}}$, $\ell_{31}
=\text {e}^{\rho_{31}}$ satisfy the triangle inequalities ($\ell_{ij}+ 
\ell_{jk}\geq \ell_{ki}$). In the triangle with these edge lengths 
denote by 
$\alpha_i^{jk}(\rho_{jk},\rho_{ij},\rho_{ki})$ the angle at vertex $i$ opposite 
to the edge of lengths $\ell_{jk}=\text{e}^{\rho_{jk}}$. Then, 
\begin{multline}\label{defV}
\hat{V}(\rho_{12}, \rho_{23},\rho_{31})= 
\rho_{23}\;\alpha_1^{23}(\rho_{23},\rho_{31},\rho_{12}) +
\rho_{31}\;\alpha_2^{31}(\rho_{31},\rho_{12},\rho_{23}) +
\rho_{12}\;\alpha_3^{12}(\rho_{12},\rho_{23},\rho_{31}) \\
\quad + 2\ELL(\alpha_1^{23}(\rho_{23},\rho_{31},\rho_{12}))+ 
2\ELL(\alpha_2^{31}(\rho_{31},\rho_{12},\rho_{23}))+ 
2\ELL(\alpha_3^{12}(\rho_{12},\rho_{23},\rho_{31})),
\end{multline}
where $\ELL(x)=-\int_0^x \log|2\sin(t)|dt$ is Milnor's Lobachevsky function.

According to the domain of definition of $\hat{V}$,
the functional ${\mathcal F}_{\frac{\pi}{2}}$ has to be considered on the domain
\begin{multline*}
 A=\{ \omega \in\R^{|E|}\ |\ \text{for all edges } [v_i,v_j] \text{ and their 
incident triangles } \Delta[v_i,v_j,v_k] \text{ there holds} \\
|w_{jk}|\text{e}^{-\omega_{jk}} + 
|w_{ki}|\text{e}^{-\omega_{ki}} \geq |w_{ij}|\text{e}^{-\omega_{ij}} \}.
\end{multline*}
Note that the functional ${\mathcal F}_{\frac{\pi}{2}}$ may be extended from the 
domain $A$ to $\R^E$ 
such that the extention is still continuously differentiable and concave, 
see~\cite[Section~4.2]{BPS13}.

\subsubsection{Conformally equivalent triangulations}
The case $\vartheta=0$ may be considered as ``dual'' to circle patterns. The 
free variables $\omega_{ij}$ are the relative counterclockwise rotation angles 
$\omega_{ij} =\arg(z_{ij})- \arg(w_{ij})$ (modulo $2\pi$) of the edges $z_{ij}$.
The dependent variables $\nu_{ij}= \log|z_{ij}|-\log|w_{ij}|$ are the changes 
of the 
logarithmic lengths of the edges. In this case, condition~\eqref{eqcondnu} 
expresses that the logarithm of the length cross-ratio does not change for any 
two incident triangles. 

Denote by $\alpha_i^{\Delta[z_i,z_j,z_k]} =
\alpha_i^{\Delta[v_i,v_j,v_k]}+ \omega_{ki}-\omega_{ij}$ the angle in the 
triangle $\Delta[z_i,z_j,z_k]$ (with counterclockwise orientation) opposite to 
the edge $z_{jk}$. As detailed in~\cite[Sec.~4.3]{BPS13}, a
corresponding concave functional is given by
\begin{align*}
{\mathcal F}_0(\omega)&= \sum\limits_{\Delta[v_i,v_j,v_k]} \Bigl(
 2\ELL(\alpha_i^{\Delta[v_i,v_j,v_k]}+ \omega_{ki}-\omega_{ij}) +
2\ELL(\alpha_j^{\Delta[v_i,v_j,v_k]}+ \omega_{ij}-\omega_{jk}) \\
&\qquad\qquad\qquad +
2\ELL(\alpha_k^{\Delta[v_i,v_j,v_k]}+ \omega_{jk}-\omega_{ki})  \\
&\qquad\qquad\qquad  + (\omega_{ki}-\omega_{ij})\log|w_{jk}| +
(\omega_{ij}-\omega_{jk})\log|w_{ki}| +
(\omega_{jk}-\omega_{ki}) \log|w_{ij}|\Bigr).
\end{align*}
The variables have to be restricted to the domain
\begin{align*}
 C&=\{ \omega\in\R^{|E|}\ |\ \text{for all vertices } v_i \text{ and their 
incident triangles } \Delta[v_i,v_j,v_k]: \\
&\qquad\qquad\qquad\quad \alpha_i^{\Delta[v_i,v_j,v_k]}+ 
\omega_{ki}-\omega_{ij}\in (0,\pi) \}.
\end{align*}

\subsection{The functional ${\mathcal F}_\vartheta(\omega)$ and its relations to 
discrete $\vartheta$-conformal maps}

Let $F:T\to\widehat{T}$ be a  discrete immersion of a simply connected 
triangulation $T$. Following our reasoning in 
Subsection~\ref{subsecvari}, we can define our new variables $\omega$ on
all edges $E$ and obtain corresponding dependent variables $\nu^\Delta(\omega)$ 
on the edges of the respective triangles $\Delta$. We assume that the variables 
$\omega$ vary within a neighborhood of (given) values which give rise to a discrete 
$\vartheta$-conformal map. That is, we consider (small) perturbations of 
a  discrete $\vartheta$-conformal map. In this way, we ensure that the 
$\nu^\Delta(\omega)$'s are well-defined and smooth. 

We will show that condition~\eqref{eqcondnu} is variational. In particular, we 
consider the 1-form
\begin{equation}\label{eqvari}
\Xi_\vartheta(\omega)=\sum_{e=[v_i,v_j]\in E_{int}} 
(\nu_{il}^{\Delta_{ilj}}-\nu_{lj}^{\Delta_{ilj}} 
+\nu_{jk}^{\Delta_{ijk}}-\nu_{ki}^{\Delta_{ijk}})d\omega_e,
\end{equation}
where $\omega\in\R^{|E|}$ denotes the vector of values of $\omega_e$ 
on the edges and the sum is taken over all interior edges $e=[v_i,v_j]$ which 
are 
contained in the two triangles $\Delta_{ilj}=\Delta[v_i,v_l,v_j]$ and 
$\Delta_{ijk}=\Delta[v_i,v_j,v_k]$ such that all edges in the triangles
are enumerated in counterclockwise orientation as in Figure~\ref{Fig2Triang}.
If the 1-form $\Xi_\vartheta$ is closed, then we can locally 
integrate: $\Xi_\vartheta=d{\mathcal F}_\vartheta$ where ${\mathcal 
F}_\vartheta(\omega)$ is some function (defined on a suitable open subset). 
Furthermore, all critical points of ${\mathcal F}_\vartheta$ 
satisfy~\eqref{eqcondnu} for all pairs 
of adjacent triangles. 

\begin{theorem}\label{theoclosed}
 The 1-form $\Xi_\vartheta$ defined by~\eqref{eqvari} in a neighborhood of a 
solution is closed. 
\end{theorem}
\begin{proof}
By~\eqref{eqvari}, the condition $d\Xi_\vartheta=0$ is equivalent to 
\[\frac{\partial (\nu_{ij}-\nu_{jk})}{\partial \omega_{ij}} 
-\frac{\partial (\nu_{jk}-\nu_{ki})}{\partial \omega_{ki}} 
=0\] 
for every triangle $\Delta[v_i,v_j,v_k]$ and cyclic permutations of the 
indices.
Differentiating~\eqref{eqtriangleom} by $\omega_{ij}$ and 
$\omega_{ki}$, we see by straightforward calculations that
\begin{multline*}
 \frac{\partial (\nu_{ij}-\nu_{jk})}{\partial \omega_{ij}} 
-\frac{\partial (\nu_{jk}-\nu_{ki})}{\partial \omega_{ki}}  \\
= \frac{\text{Im}\left(w_{ij}\cdot i\text{e}^{i\vartheta}\cdot 
\text{exp}(\text{e}^{i\vartheta}(i\omega_{ij}+
(\nu_{ij}-\nu_{jk})))\cdot \overline{w_{ki}}\cdot
\text{e}^{-i\vartheta}\cdot 
\text{exp}(\text{e}^{-i\vartheta}(-i\omega_{ki} +
(\nu_{ki}-\nu_{jk})))\right)}{\text{Im}
\left(w_{ki}\cdot \text{e}^{i\vartheta}\cdot 
\text{exp}(\text{e}^{i\vartheta}(i\omega_{ki}+
(\nu_{ki}-\nu_{jk})))\cdot \overline{w_{ij}}\cdot
\text{e}^{-i\vartheta}\cdot 
\text{exp}(\text{e}^{-i\vartheta}(-i\omega_{ij} + 
(\nu_{ij}-\nu_{jk})) \right)} \\
\quad +\frac{\text{Im}\left(w_{ij}\cdot\text{e}^{i\vartheta}\cdot 
\text{exp}(\text{e}^{i\vartheta}(i\omega_{ij}+ 
(\nu_{ij}-\nu_{jk}))) \overline{w_{ki}}\cdot
(-i)\text{e}^{-i\vartheta}\cdot 
\text{exp}(\text{e}^{-i\vartheta}(-i\omega_{ki}+
(\nu_{ki}-\nu_{jk})))\right)}
{\text{Im}\left(w_{ki}\cdot \text{e}^{i\vartheta}\cdot 
\text{exp}(\text{e}^{i\vartheta}(i\omega_{ki}+ 
(\nu_{ki}-\nu_{jk})))\cdot \overline{w_{ij}}\cdot
\text{e}^{-i\vartheta}\cdot 
\text{exp}(\text{e}^{-i\vartheta}(-i\omega_{ij}+
(\nu_{ij}-\nu_{jk})))\right) } \\
=
\frac{\text{Im}\left(w_{ij}\overline{w_{ki}}\text{exp}(\text{e}^{i\vartheta}
(i\omega_{ij}+ (\nu_{ij}-\nu_{jk}))+ 
\text{e}^{-i\vartheta} (i\omega_{ki}+
(\nu_{ki}-\nu_{jk})))
(i\text{e}^{
i\vartheta}\cdot \text{e}^{-i\vartheta}- i\text{e}^{-i\vartheta}\cdot 
\text{e}^{i\vartheta})\right)}{\text{Im}\left( w_{ki}\overline{w_{ij}}
\text{exp}(\text{e}^{i\vartheta}(i\omega_{ki}+
(\nu_{ki}-\nu_{jk})) +\text{e}^{-i\vartheta}(-i
\omega_{ij}+ (\nu_{ij}-\nu_{jk})))\right) } \\
=0.
\end{multline*}
This shows that the 1-form $\Xi_\vartheta$ is closed.
\end{proof}

\begin{theorem}\label{theoconcave}
The function ${\mathcal F}_\vartheta$ which is a local integral of $\Xi_\vartheta$ 
is locally concave, that is, 
the second derivative is a negative semidefinite quadratic form. Its 
one-dimensional kernel is  spanned by $(1, 1,\dots, 1)$.
\end{theorem}
\begin{proof}
The calculations in the proof of Theorem~\ref{theoclosed} show in 
fact that 
$\frac{\partial (\nu_{ij}-\nu_{jk})}{\partial 
\omega_{ij}}= \cot \alpha_i^{\widetilde{\Delta}_{ijk}}$, 
where $\alpha_i^{\widetilde{\Delta}_{ijk}}(\omega)$ is the (positively 
oriented) angle in the triangle $\widetilde{\Delta}_{ijk}=\Delta[z_i,z_j,z_k]$ 
corresponding to the given values of $\omega$. 
Furthermore, by similar calculations we obtain 
$\frac{\partial (\nu_{ij}-\nu_{jk})}{\partial 
\omega_{ki}}=-\cot \alpha_k^{\widetilde{\Delta}_{ijk}} -\cot 
\alpha_i^{\widetilde{\Delta}_{ijk}}$.
This shows by the same arguments as in~\cite[Prop.~4.2.4]{BPS13}, that the 
Hessian of the functional ${\mathcal F}_\vartheta$ is a negative semidefinite 
quadratic form 
$-\sum_{e \cong \tilde{e}} \cot \alpha^{e,\tilde{e}} 
(d\omega_e-d\omega_{\tilde{e}})^2$ where the sum is taken over all pairs of 
edges $e,\tilde{e}$ which are incident to the same triangular face and 
$\alpha^{e,\tilde{e}}$ is the angle between these edges. The one-dimensional 
kernel is spanned by $(1, 1,\dots, 1)$. 
\end{proof}

It is interesting to note that the entries 
of the second derivative of ${\mathcal F}_\vartheta$ (at $\omega=0$) are the same (up to sign) as 
in 
the known cases for circle 
patterns and discrete conformal equivalence, see for example~\cite{BPS13}.

We have now proved the first part of Theorem~\ref{theovari}. In order to use 
the variational principle for computational purposes, it is important that 
maximizers $\omega$ of ${\mathcal F}_\vartheta$ also correspond to 
discrete $\vartheta$-conformal maps, which can be obtained from the values of 
$\omega$ and $\nu(\omega)$ by gluing the corresponding 
triangles according to the given combinatorics.

\begin{theorem}\label{theoconstruct}
 Every maximizer of ${\mathcal F}_\vartheta$ on a simply connected triangulation 
gives rise to a $\vartheta$-conformal map.
\end{theorem}
\begin{proof}
 Note that we implicitly assume that every maximizer $\omega$ lies in the 
``allowed domain'', that is, for every triangle there exist values for the 
depending parameters $\nu$ to the corresponding values of $\omega$. So locally, 
$\omega$ gives rise to a realization of triangles which is an orientation preserving 
embedding.

We prove the claim by successive construction of an immersed triangulation with 
the same combinatorics as the given one and such that~\eqref{eqdefdiscconf} 
holds for all pairs of incident triangles.
Start with any triangle and its realization, that is suitable values 
for $\nu$ on the edges such that~\eqref{eqtriangleom} holds. As the values of 
$\nu$ are unique up to a common additive constant, this corresponds to fixing a 
similarity transformation.

Next we consider an incident triangle. We know by assumption, that suitable 
dependent parameters $\nu$ exist on its edges and fix the freedom for $\nu$ by 
choosing $\nu_{ij}$ to agree with the value from the first triangle on the 
common edge $[v_i,v_j]$. This uniquely determines the remaining two values for 
$\nu$ and thus the embedding of the second triangle. Note that by assumption 
the values of $\nu$ on the boundary edges of the pair of incident triangles 
satisfy~\eqref{eqcondnu}. We can successively continue  this procedure and 
obtain further values for $\nu$ and a sequence of embedded triangles according 
the given combinatorics as long as this sequence never closes up. In this case, 
there might appear two different values for $\nu$ on the same edge (i.e.\ the 
corresponding triangles cannot be glued together).

By assumption, the given triangulation is simply connected. Therefore, it is 
sufficient to show that no ambiguities occur if we consider the sequence of 
incident triangles $\Delta[v_0,v_j,v_{j+1}]$, $j=1,\dots,n$, $n+1=1$, which 
build a flower. Assume without loss of generality that we start with 
$\Delta[v_0,v_1,v_2]$ and given $\nu$'s on the edges and then successively 
determine the value of the $\nu$'s on the other edges. All these values are 
well-defined except possibly the value on the edge $[v_0,v_1]$. Here, 
$\nu_{0,1}$ is initially given, but following our construction gives some value 
$\nu_{0,1}^*$ in the triangle $\Delta[v_0,v_n,v_{1}]$ which might be different. 
But we can sum up condition~\eqref{eqcondnu}, which holds for all pairs of 
incident triangles, in particular 
\begin{align*}
\nu_{0,j}-\nu_{j,j+1} 
+\nu_{j+1,j+2}-\nu_{0,j+2}&=0\qquad \text{for }j=1,\dots, n-2,\\  
\nu_{0,n-1}-\nu_{n-1,n} +\nu_{n,1}-\nu_{0,1}^*&=0 \\
\nu_{0,n}-\nu_{n,1} +\nu_{1,2}-\nu_{0,2}=0. 
\end{align*}
Adding up all these equations leads to 
$\nu_{0,1}-\nu_{0,1}^*=0$. This proves that the values of $\nu$ are 
well-defined on a flower. 

Therefore, starting with any triangle, we can successively determine the values 
of $\nu$ on all edges and thus successively embed all triangles. This gives a 
discrete immersion $F$. Condition~\eqref{eqcondnu} implies that $F$ is also 
$\vartheta$-conformal.
\end{proof}

\subsection{A variational principle with variables at vertices}

For our variational principle explained above, we start from real parameters 
$\omega_e$ on the 
{\em edges} of the triangles (which then determine $\nu^\Delta(\omega)$ 
by~\eqref{eqtriangleom} and finally lead to an immersed
triangulation). Alternatively, one can consider real 
parameters $u_i$ on the {\em vertices} of the triangles. This is similar to the 
logarithmic scale factors considered in~\cite{BPS13} and Remark~\ref{remlog}.
To every edge $[v_i,v_j]$ we then associate the arithmetic mean 
$(u_i+u_j)/2$. 
Denoting the dependent parameters by $\xi=\xi(u)$, we can 
reformulate~\eqref{eqtriangleom} as
\begin{multline}\label{equxi}
w_{ij}\cdot \text{exp}({\text{e}^{i\vartheta}((u_i+u_j)/2+ 
i(\xi_i+\xi_j)/2}))
+w_{jk}\cdot \text{exp}({\text{e}^{-i\vartheta}((u_j+u_k)/2+
i(\xi_j+\xi_k)/2}))\\
+w_{ki}\cdot \text{exp}({\text{e}^{-i\vartheta}((u_k+u_i)/2+ 
i(\xi_k+\xi_i)/2})) =0.
\end{multline}
Analogously as for $\nu(\omega)$, these dependent values $\xi(u)$ are only 
unique up to a common constant and depend on all values of $u$ in the 
triangle $\Delta[v_i,v_j,v_k]$. 

We will now reformulate the condition for discrete $\vartheta$-conformal maps 
in these variables. To this end,
consider a flower as for example in Figure~\ref{Figdefalpha}. Assume that  
values $u_j$ 
are given on the center of the flower $v_0$ and all its incident vertices 
$v_1,\dots,v_N$ such that there exist a solution of~\eqref{equxi} in every 
triangle. Similarly as in~\eqref{eqdefalpha}, the oriented angle between the 
edges $[z_0,z_j]$ and $[z_0,z_{j+1}]$ is 
\[\alpha_{j,j+1}^{\Delta[z_0,v_j,v_{j+1}]}=  
\alpha_{j,j+1}^{\Delta[v_0,v_j,v_{j+1}]} 
+\textstyle\frac{1}{2} (u_{j+1}-u_j)\sin\vartheta+ \textstyle\frac{1}{2} 
(\xi_{j+1}^{\Delta_{0,j,j+1}}-\xi_j^{\Delta_{0,j,j+1}})\cos\vartheta.\]
Recall that the dependent parameters 
$\xi_{j+1}^{\Delta_{0,j,j+1}}-\xi_j^{\Delta_{0,j,j+1}}$ are defined locally in 
the triangles $\Delta_{0,j,j+1}=\Delta[v_0,v_j,v_{j+1}]$ at the vertex $v_0$. 
Note that  $\xi_{j+1}^{\Delta_{0,j,j+1}}-\xi_j^{\Delta_{0,j,j+1}}$
smoothly depends on the given values $u_j,u_{j+1},u_0$ in 
a neighborhood of a solution.
Then this flower can only be embedded if the following closing condition 
holds:
\begin{equation}\label{eqmono2}
 0=\sum_{j=1}^6 (\text{e}^{-i\vartheta}(u_{j+1}-u_j)+ 
i\text{e}^{-i\vartheta} 
(\xi_{j+1}^{\Delta_{0,j,j+1}}-\xi_j^{\Delta_{0,j,j+1}})= 
i\text{e}^{-i\vartheta}\sum_{j=1}^6 
(\xi_{j+1}^{\Delta_{0,j,j+1}}-\xi_j^{\Delta_{0,j,j+1}}),
\end{equation}
where we identify indices modulo $6$. 
Condition~\eqref{eqmono2} is variational, if for any two 
adjacent triangles $\Delta[v_i,v_l,v_j]$ and $\Delta[v_i,v_j,v_k]$ we have
\begin{align*}
 \frac{\partial (\xi_{k}^{\Delta_{i,j,k}}-\xi_j^{\Delta_{i,j,k}})}{\partial 
u_j} +
\frac{\partial (\xi_{j}^{\Delta_{i,l,j}}-\xi_l^{\Delta_{i,l,j}})}{\partial 
u_j} &=
 \frac{\partial (\xi_{i}^{\Delta_{i,j,k}}-\xi_k^{\Delta_{i,j,k}})}{\partial 
u_i} +
\frac{\partial (\xi_{l}^{\Delta_{i,l,j}}-\xi_i^{\Delta_{i,l,j}})}{\partial 
u_i}.
\end{align*}
This in turn is true as straightforward calculations (for example using a 
computer algebra system) show that if $\Delta_{1,2,3}=\Delta[v_1,v_2,v_3]$ is 
an embedded 
non-degenerate triangle with angles $\alpha_1,\alpha_2,\alpha_3\in(0,\pi)$ at 
the corresponding vertices, we have
\begin{align*}
 \frac{\partial}{\partial u_3}(\xi_{3}^{\Delta_{1,2,3}} 
-\xi_2^{\Delta_{1,2,3}})/2 &=\cot \alpha_2 =\frac{\partial}{\partial 
u_1}(\xi_{2}^{\Delta_{1,2,3}}-\xi_1^{\Delta_{1,2,3}})/2.
\end{align*}
This also proves that the Hessian of the functional is again the 
cotan-Laplacian as it is the case for conformal equivalence and for circle 
patterns. Furthermore, this implies  that the 
corresponding functional is locally concave and therefore the maximizers
are unique (up to a global transformation).

\paragraph{Brief review on known cases.}
For the known case of conformally equivalent triangulations 
($\vartheta=0$), the variables $u_i$ are logarithmic scale factors and the 
dependent variables $\xi_i$ correspond to relativ rotations. A corresponding 
concave functional is studied in~\cite[Section~4]{BPS13} and given by 
\begin{align*}
 {\mathcal E}_0(u)&=  -\sum\limits_{\Delta[v_i,v_j,v_k]} \Bigl(
\hat{V}(\log|w_{ij}|+(u_i+u_j)/2, \log|w_{jk}|+(u_j+u_k)/2, 
\log|w_{ki}|+(u_k+u_i)/2) \\
&\qquad\qquad -\pi(u_i+u_j+u_k)/2 
 -\frac{\pi}{2}(\log|w_{ij}| +\log|w_{jk}|+ \log|w_{ki}|+ u_i+u_j+u_k)\Bigr) \\
&\quad - \sum\limits_{v_i} \Theta_i u_i/2,
\end{align*}
where the first sum is taken over all triangles and the 
second sum over all vertices of the given triangulation $T$. 
Note that~\cite{BPS13} investigates twice the negative of this functional (that is 
$E_{T,\Theta,\lambda}(u)= -2 {\mathcal E}_0(u)$).
The function $\hat{V}$ has been defined in~\eqref{defV}. The numbers $\Theta_i$ 
are the desired angles sums at vertices of the image triangulation, so 
$\Theta_i=2\pi$ in our case for all interior vertices $v_i\in V_{int}$. For 
boundary vertices we either fix $\Theta_i$ or $u_i$, 
see~\cite[Section~3.1]{BPS13}.
Note that the functional ${\mathcal E}_0$ is only defined on a part of $\R^{|V|}$, namely where all 
triangle inequalities for the new edge lengths (given by $|w_{ij}|\text{e}^{(u_i+u_j)/2}$) hold. 
But this functional may be extended to a concave continously differentiable functional on 
$\R^{|V|}$, see~\cite[Prop.~4.1.5]{BPS13}.

For the known case of circle patterns ($\vartheta=\pi/2$), the variables $u_i$ 
are the relative counterclockwise rotation angles of the edge stars at $v_i$ and 
the dependent variables $\xi_i$ are logarithmic scale factors. A corresponding 
concave functional is studied in~\cite[Section~6.2]{BS02} and given by
\begin{align*}
 {\mathcal E}_{\frac{\pi}{2}}(u)=  \sum\limits_{\Delta[v_i,v_j]} \bigl(
\ELL(\alpha_k^{\Delta[v_i,v_j,v_k]}+(u_j-u_i)/2) 
+\ELL(\alpha_l^{\Delta[v_i,v_l,v_j]}+(u_i-u_j)/2) \bigr),
\end{align*}
where the sum is taken over all interior edges $[v_i,v_j]$ of $T$ with 
incident, counterclockwise oriented triangles $\Delta[v_i,v_j,v_k]$ and 
$\Delta[v_i,v_l,v_j]$. Note that also this functional is not defined for all possible values of 
$u\in\R^{|V|}$, as we need all angles $\alpha_k^{\Delta[v_i,v_j,v_k]}+(u_j-u_i)/2$ to only take 
values in $(0,\pi)$.

\section{Concluding remarks}

\paragraph{Relations to Ronkin functions, amoebas and coamoebas}
 It is interesting to note that the functional $\hat{V}$ defined in~\eqref{defV}, which reappears 
in the formula for ${\mathcal E}_0$, is the Ronkin function of the linear polynomial $z_1+z_2+z_3$, 
see~\cite[Remark~4.2.7]{BPS13} and the references therein. The domain of definition of $\hat{V}$ is 
an amoeba. 
This notion was introduced in~\cite{GKZ94} by Gelfand, Kapranov and Zelevinsky. The amoeba of a 
complex polynomial $p(z_1,\dots,z_n)$ is defined as the image of the set of zeros of $p$ under the 
map $(z_1,\dots,z_n)\mapsto (\log|z_1|,\dots,\log|z_n|)$. A coamoeba, also 
called Newton polygon or Newton polytope, constitutes a closely related ``dual'' notion, 
see for example~\cite{Pa08} for a nice introductory presentation. Similarly as for amoeba and 
coamoeba, our considerations are based on logarithmic coordinates. Unfortunately, it is still 
unclear why the Ronkin function is related to some of the known functionals and how to connect the 
known explicit formulas for the functionals detailed in Section~\ref{SecReview} into a common 
framework.

\paragraph{Relations to hyperbolic polyhedra and their duals}
Recall that discretely conformally equivalent trianguations are described intrinsically via edge 
lengths. Furthermore, this approach is related to hyperbolic polyhedra, see~\cite{BPS13,Sp17}. But 
it was also noted soon that the theory of discrete conformal equivalence requires changes in the 
combinatorics or else faces the degeneration of triangles (as indicated in the proofs of 
Lemma~\ref{lempart} and Lemma~\ref{lemdiscconf}).

On the other hand, circle pattern are defined using intersection angles which are not intrinsic 
variables. If we stereographically project a circle pattern to the sphere and regard 
this sphere as the boundary of hyperbolic space, we can associate to every circle a hyperbolic 
plane. The interior intersection angles between intersecting circles agree with the dihedral angles 
between the corresponding intersecting hyperbolic planes.  
The ideal hyperbolic polyhedron which is bounded by these planes is determined 
by the dihedral angles between these planes. Instead, one can consider the 
dual polyhedron as defined by Rivin in~\cite{HR93,Ri96} using a suitable Gauss-map. Rivin then 
characterized this dual polyhedron via its metric. 

It remains an open question how the one 
parameter family of discrete $\vartheta$-conformal maps may be related to a corresponding family of 
polyhedra which interpolates between the two known cases.

\subsection*{Funding}
This work was supported by the DFG Collaborative Research Center
TRR~109 ``Discretization in Geometry and Dynamics''.

\section*{Acknowledgments}
The author is grateful to Alexander Bobenko for initiating her research on
conformally symmetric triangulations. 

Furthermore, it is a pleasure to thank 
Boris Springborn, Wai Yeung Lam, and, in particular, Niklas Affolter for 
interesting discussions on discrete $\vartheta$-conformal maps and their 
connections to circle patterns and discrete conformal equivalence. Interesting 
insight into the variational principle for the discrete $\vartheta$-conformal 
maps has been gained during the inspiring Young Investigator's Workshop of the 
DFG Collaborative Research Center TRR~109 ``Discretization in Geometry and 
Dynamics'' in March 2018.

This research was supported by the DFG Collaborative Research Center
TRR~109 ``Discretization in Geometry and Dynamics''.
%
%

%

\bibliographystyle{amsplain}
\bibliography{ConfSymmetric}

\end{document}